\documentclass{amsart}

\usepackage{amssymb}
\usepackage{amsthm}
\usepackage{amsmath}
\usepackage{enumitem}
\usepackage[mathscr]{euscript}

\usepackage{cancel}
\usepackage{soul}

\usepackage[normalem]{ulem}

\usepackage[usenames]{color}

\usepackage{hyperref}

\theoremstyle{plain}
\newtheorem{proposition}{Proposition}[section]
\newtheorem{theorem}[proposition]{Theorem}
\newtheorem{lemma}[proposition]{Lemma}
\newtheorem{corollary}[proposition]{Corollary}
\theoremstyle{definition}

\newtheorem{definition}[proposition]{Definition}
\newtheorem{observation}[proposition]{Observation}
\theoremstyle{remark}
\newtheorem{remark}[proposition]{Remark}

\newtheorem{question}[proposition]{Question}

\DeclareMathOperator{\Aut}{Aut}

\DeclareMathOperator{\supp}{supp}

\DeclareMathOperator{\id}{id}

\DeclareMathOperator{\injrad}{inj-rad}

\DeclareMathOperator{\dist}{d}

\DeclareMathOperator{\Cc}{\mathcal{C}}
\DeclareMathOperator{\Fc}{\mathcal{F}}

\DeclareMathOperator{\Nc}{\mathcal{N}}
\DeclareMathOperator{\Oc}{\mathcal{O}}

\DeclareMathOperator{\Tc}{\mathcal{T}}
\DeclareMathOperator{\Uc}{\mathcal{U}}

\DeclareMathOperator{\Vc}{\mathcal{V}}
\DeclareMathOperator{\Pc}{\mathcal{P}}
\DeclareMathOperator{\Sc}{\mathcal{S}}

\DeclareMathOperator{\Cb}{\mathbb{C}}

\DeclareMathOperator{\Nb}{\mathbb{N}}

\DeclareMathOperator{\Rb}{\mathbb{R}}
\DeclareMathOperator{\Sb}{\mathbb{S}}

\newcommand{\abs}[1]{\left|#1\right|}
\newcommand{\norm}[1]{\left\|#1\right\|}

\newcommand{\ip}[1]{\left\langle #1\right\rangle}

\begin{document}

\title[Families of strongly pseudoconvex domains]{Smooth equivalence of families of strongly pseudoconvex domains}

\author[]{Herv\'{e} Gaussier,$^1$ Xianghong Gong,$^2$ \and Andrew Zimmer$^3$}
\address{H. Gaussier:
Univ. Grenoble Alpes, CNRS, IF, F-38000 Grenoble, France
}
\email{herve.gaussier@univ-grenoble-alpes.fr}
\address{X. Gong:
Department of Mathematics,
University of Wisconsin-Madison, Madison, WI 53706, U.S.A.}
\email{gong@math.wisc.edu}

\address{A. Zimmer:
Department of Mathematics,
University of Wisconsin-Madison, Madison, WI 53706, U.S.A.}
\email{amzimmer2@wisc.edu}

\thanks{$^1$Partially supported by  ERC ALKAGE. \\
\indent $^2$Partially supported by  NSF grant DMS-2054989. \\
\indent $^3$Partially supported by a Sloan research fellowship and NSF grants DMS-2105580 and DMS-2104381.}
\date{\today}


\begin{abstract}

We establish a smoothness result for families of biholomorphisms between smooth families of strongly pseudoconvex domains, each with trivial biholomorphism group. This is accomplished by considering the Riemannian geometry of their Bergman metrics and proving a result about the smoothness of families of isometries between smooth families of Riemannian manifolds.

\end{abstract}

\maketitle

\sloppy

\section{Introduction}

In this paper we study families of complex manifolds $\mathcal{M}=\{M_t\}$ that depend smoothly on a parameter $t\in \Tc$. This is a rich theory with many deep results including Kodaira--Spencer's deformation theory for compact complex manifolds and complex spaces with singularity (see~\cite{GLS-deformation, Kodaira-deformation} and the references therein), Newlander--Nirenberg's~\cite{NN} and Nijenhuis--Woolf's~\cite{NW} works on deformation of complex structures, and Hamilton's~\cite{Hamilton3} Newlander--Nirenberg type theorem for families of compact domains with boundary.

We are particularly interested in studying the regularity properties of  families of biholomorphisms  between two families of complex manifolds and more specifically the following question.

\begin{question}\label{main question} Suppose $\mathcal{M}=\{M_t\}$ and $\hat{\mathcal{M}}=\{\hat{M}_t\}$ are two families of complex manifolds  that depend smoothly on a parameter $t\in \Tc$ (where $\Tc$ is an open set in a smooth manifold). If for every $t \in \Tc$ the  manifolds   $M_t$ and $\hat{M}_t$ are biholomorphic, then is it possible to find a family of biholomorphisms
$\left\{F(\cdot,t) : M_t \rightarrow \hat{M}_t\right\}$  such that
 $F$ is smooth?
\end{question}

One can also ask the same question for other structures on manifolds, like smooth families of Riemannian metrics and families of  isometries instead of biholomorphisms.

In this paper, we consider strongly pseudoconvex domains in complex Euclidean space. A smooth family of strongly pseudoconvex domains can be defined precisely as follows.

\begin{definition}\label{defn: families of str pconvex domains}
A family $\{\Omega_t\}_{t \in \mathcal T}$ of smoothly bounded strongly pseudoconvex domains in $\mathbb C^d$ is \emph{smooth} if for every $t_0 \in \Tc$ there are  a neighborhood $\Tc_0$ of $t_0$ in $\Tc$,  a neighborhood $\mathcal U$ of $\overline{\Omega_{t_0}}$,
 and a $\mathcal C^{\infty}$  smooth map $\Phi : \mathcal U\times\mathcal T_0 \rightarrow \mathbb C^d$ such that for every $t \in \mathcal T_0$ the map $\Phi(\cdot, t) : \mathcal U \rightarrow \Cb^d$ is a diffeomorphism onto its image and $\Phi(\cdot, t)(\overline{\Omega_{t_0}}) = \overline{\Omega_t}$.
\end{definition}

Previously, the first two authors~\cite{GG2020}  answered  Question~\ref{main question} affirmatively in the special case of smoothly bounded domains in $\Cb$ (where every smoothly bounded domain is strongly pseudoconvex). They also considered families of strongly pseudoconvex domains in higher dimensions,   and using techniques from several complex variables proved   results about the continuity of families of biholomorphisms.

In this paper, we provide a complete answer to Question~\ref{main question} for rigid domains (i.e. domains with trivial biholomorphism group). Our approach is geometric in nature and our result for strongly pseudoconvex domains will be a consequence of a general result about smoothly varying families of Riemannian metrics. To apply this general result to our
 setting, we use   the Bergman metrics of the domains. Our main   results are  as follows.

\begin{theorem}[see Theorem~\ref{thm:smoothness of family of biholomorphisms in paper} below]
\label{thm:smoothness of family of biholomorphisms}
Suppose $\{\Omega_t\}_{t \in \Tc}$ and $\{\hat{\Omega}_t\}_{t \in \Tc}$ are two smooth families of strongly pseudoconvex domains in $\Cb^d$. If for each $t \in \Tc$
\begin{enumerate}\renewcommand{\labelenumi}{$(\alph{enumi})$}
  \item the biholomorphism group $\mathsf{Aut}(\Omega_t)$ is trivial, and
\item there  is a biholomorphism $F_t :\Omega_t \rightarrow \hat{\Omega}_t$,
\end{enumerate}
then the map
$$
(p,t)  \mapsto F_t(p)
$$
is smooth.
\end{theorem}

\begin{remark} Once smoothness is known on the interior, results in~\cite{GG2020} imply that the map $(p,t)  \mapsto F_t(p)$ extends to a smooth map on the closure
$$
\bigcup_{t \in \Tc} \overline{\Omega_t} \times \{t\} \subset \Cb^d \times \Tc.
$$
The results in~\cite{GG2020} were proved for $\Tc=[0,1]$, however, the proofs hold for any open set $\Tc$ in a smooth manifold.
\end{remark}

The triviality assumption on the biholomorphism groups is crucial in Theorem~\ref{thm:smoothness of family of biholomorphisms}. Indeed, in Section~\ref{sec: examples and conjecture} we show that any strongly pseudoconvex domain with non-trivial automorphism group is contained in two smooth families of strongly pseudoconvex domains which are pairwise biholomorphic, but there is no continuous family of biholomorphisms between the families. More precisely, we prove the following.

 \begin{proposition}[see Proposition~\ref{Omega 1/k} below]\label{Omega 1/k in intro}
Let $D$ be a strongly pseudoconvex domain in $\mathbb C^d$. If $\Aut(D) \neq \{\id_D\}$, then there are two smooth families  $\{\Omega_t\}_{t \in (-1,1)}$, $\{\hat \Omega_t\}_{t \in (-1,1)}$ of strongly pseudoconvex
  domains where $\Omega_0 = D = \hat \Omega_0$, $\Omega_t$ is biholomorphic to $\hat \Omega_t$ for all $t \in (-1,1)$,  and any family of biholomorphisms $\{ F_t : \Omega_t \rightarrow \hat \Omega_t\}_{t \in (-1,1)}$ is discontinuous at $t=0$.
\end{proposition}

Our geometric approach also allows  us to consider Question~\ref{main question}  in the context of families of Riemannian metrics. For compact manifolds, we prove the following analogue of Theorem~\ref{thm:smoothness of family of biholomorphisms}.

\begin{theorem}[see Theorem~\ref{thm:riemannian cpct case in paper} below]\label{thm:riemannian cpct case}
Suppose $M$ and $\hat{M}$ are two compact smooth manifolds. Assume $\{ g_t\}_{t \in \Tc}$ and $\{ \hat{g}_t\}_{t \in \Tc}$ are two smooth families of Riemannian metrics on $M$ and $\hat{M}$, respectively. If for each $t \in \Tc$,
\begin{enumerate}\renewcommand{\labelenumi}{$(\alph{enumi})$}
\item the isometry group $\mathsf{Isom}(M,g_t)$ is trivial, and
\item there is an isometry $F_t : (M, g_t) \rightarrow (\hat{M}, \hat{g}_t)$,
\end{enumerate}
then the map
$$
(p,t) \in M \times \Tc \longmapsto F_t(p) \in \hat{M}
$$
is smooth.
\end{theorem}

It seems likely that a version of Proposition~\ref{Omega 1/k in intro} also holds in the Riemannian case, but we do not pursue such matters in this paper.

The two theorems seem of different flavor: Theorem~\ref{thm:smoothness of family of biholomorphisms} deals with families of strongly pseudoconvex domains while Theorem~\ref{thm:riemannian cpct case} deals with families of compact Riemannian manifolds. However, both theorems are consequences of the following general theorem about smooth families of real analytic Riemannian metrics.

\begin{theorem}[see Theorem~\ref{thm:continuity implies smoothness in riemannian case} below]\label{thm:continuity implies smoothness in riemannian case in intro}
Suppose $M$ and $\hat{M}$ are two connected smooth manifolds. Assume $\{ g_t\}_{t \in \Tc}$ and $\{ \hat{g}_t\}_{t \in \Tc}$ are two smooth families of complete Riemannian metrics on $M$ and $\hat{M}$, respectively. If
\begin{enumerate}[label=$(\alph*)$]
\item\label{item:taming infinity} there exists a relatively compact connected open set $U \subset M$
such that the inclusion map $\iota : U \hookrightarrow M$ induces   a surjection
  $\iota_*: \pi_1(U,u_0) \rightarrow \pi_1(M,u_0)$ of fundamental groups for some (and hence any) $u_0 \in U$,
\item for each $t \in \Tc$,
\begin{enumerate}[label=$(\roman*)$]
\item\label{item:each gt is real analytic} the manifold $M$ has a real analytic structure for which the metric $g_t$ is real analytic,
\item the isometry group $\mathsf{Isom}(M,g_t)$ is discrete, and
\item  there is an isometry $F_t : (M, g_t) \rightarrow (\hat{M}, \hat{g}_t)$,
\end{enumerate}
and
\item the map
$$
(p,t) \in M \times \Tc \longmapsto F_t(p) \in \hat{M}
$$
is continuous,
\end{enumerate}
then the map
$$
(p,t) \in M \times \Tc \longmapsto F_t(p) \in \hat{M}
$$
is smooth.
\end{theorem}

\begin{remark} In part~\ref{item:each gt is real analytic}, the real analytic structure may depend on $t$. We also note that since $F_t$ is an isometry, we can
 use $F_t$ to define a real analytic structure on $\hat{M}$ for which $\hat{g}_t$ is real analytic.
\end{remark}

 The case of noncompact manifolds generates technical difficulties coming from the asymptotic behavior of the manifold and metrics. One can view  condition~\ref{item:taming infinity} as a hypothesis which tames these difficulties. Further, for a smoothly bounded domain $\Omega \subset \Cb^d$, this condition clearly holds since  there is a  deformation retraction from $\Omega$ onto a relatively compact subdomain, obtained by flowing along the inward pointing normal lines.

To reduce Theorem~\ref{thm:smoothness of family of biholomorphisms} to Theorem~\ref{thm:continuity implies smoothness in riemannian case in intro},
 we use the Bergman metrics of the domains, which are complete, real analytic, and invariant under the biholomorphism group. Although the domains vary with respect to $t$, the situation can be translated to families of metrics on a central domain $\Omega_{t_0}$ by using the map $\Phi$ in Definition~\ref{defn: families of str pconvex domains}. Further, from~\cite[Theorem 1.17]{GreeneKrantz} for a strongly pseudoconvex domain, isometries of the Bergman metric are either holomorphic or antiholomorphic. Hence, if the biholomorphism group is trivial, then the isometry group of the Bergman metric is discrete (in fact has
at most
two elements).

Deep results of R. Hamilton, see Appendix~\ref{sec:smoothness of Bergman kernels}, imply that the Bergman metrics vary smoothly over smooth families of strongly pseudoconvex
domains. Finally, it was proved  in~\cite{GG2020} that the family of biholomorphisms $\{F_t\}_t$ is continuous with respect to $t$. It follows then from Theorem~\ref{thm:continuity implies smoothness in riemannian case in intro} that the map $(p,t) \mapsto F_t(p)$ is a smooth map.

To reduce Theorem~\ref{thm:riemannian cpct case} to Theorem~\ref{thm:continuity implies smoothness in riemannian case in intro}, we apply the Ricci flow for a small fixed time to each metric, which produces a real analytic metric with the same isometries. Using Hamilton's proof of the short-time existence for the Ricci flow, we verify that the new family of metrics is also smooth. We can then apply Theorem~\ref{thm:continuity implies smoothness in riemannian case in intro} as soon as we prove the continuity of the map $(p,t) \mapsto F_t(p)$. Since the manifolds are compact, this is a simple consequence of the Arzel\`a-Ascoli Theorem.

\section{Preliminaries}

\subsection{Notations}

In this section we fix any possibly ambiguous notation.

\begin{enumerate}[label=(\alph*)]
\item The parameter space $\mathcal T$ is always a connected
 open set in a smooth manifold.

\item We let $\mathbb N=\{0,1,2,\dots,\}$.

\item For a relatively compact domain $U$ in $\Rb^d$  and a function $f$ that is $C^k$ on $\overline U$, we define
\begin{equation}\label{eqn:Ck norm}
|f|_{U,k}=\max_{i_1+\cdots+i_d\leq k}\, \sup_{x\in U}\left|\frac{\partial^{i_1+\cdots i_d} f}{\partial x_1^{i_1}\cdots\partial x_d^{i_d}}(x)\right|.
\end{equation}

\item  In this paper, we work in the $C^\infty$ setting: smooth always means $C^\infty$-smooth, Riemannian metrics are always assumed to be at least $C^\infty$-smooth (and sometimes real analytic), and strongly pseudoconvex domains are always assumed to have $C^\infty$-smooth boundary.
\end{enumerate}

\subsection{Riemannian manifolds}   In this subsection we recall some basic concepts from Riemannian geometry.

Given a complete Riemannian manifold $(M,g)$ we will let $\dist_g$ denote the distance induced by $g$ and let $B_g(p,r) \subset M$ denote the open metric ball of radius $r>0$ centered at $p \in M$ with respect to $\dist_g$. Given $p \in M$, we let
$$
B_{T_pM}(0,r) :=\left\{ v \in T_p M : g_p(v,v) < r^2\right\}
$$
denote the open ball of radius $r > 0$ centered at $0$ in the inner product space $(T_p M, g_p)$.

Throughout the paper, we let   $\exp^g_p : T_p M \rightarrow M$ denote  the {\it exponential map} at $p \in M$ associated to $g$, which is the map such that for any $v \in T_p M$ the curve
$$
t \in \Rb \mapsto \exp^g_p(tv) \in M
$$
is the unique geodesic through $p$ with initial velocity $v$. The \emph{injectivity radius at $p$}, denoted $\injrad_{g}(p)$, is the supremum of all $r>0$ such that the map
$$
\exp^g_{p}\colon B_{T_p M}(0, r)\to B_g(p,r)
$$
is a diffeomorphism.

Geodesics in a Riemannian manifold satisfy a non-linear second order differential equation whose coefficients depend  on the metric $g$ (see \cite[p.~103]{LeeRiemGeom} for the explicit equation). So we have the following.

\begin{observation}\label{obs:smoothness of exponential maps} Suppose $\{g_t\}_{t \in \Tc}$ is a smooth family of complete Riemannian metrics on a manifold $M$. Then the map
$$
(v, t) \in TM \times \Tc \longmapsto \exp_{\pi(v)}^{g_t}(v) \in M
$$
is smooth (where $\pi(v) \in M$ is the basepoint of $v$).
\end{observation}

Given two Riemannian manifolds $(M,g)$ and $(N,h)$, a smooth map $f : M \rightarrow N$ is a \emph{local isometry} if  $f^* h = g$. If, in addition, $f$ is a diffeomorphism, then $f$ is an \emph{isometry}. An isometry $f : M \rightarrow N$ maps geodesics to geodesics and hence for any $p \in M$ we have
\begin{equation}\label{efp}
\exp^h_{f(p)} \circ df_p = f \circ \exp^g_p.
\end{equation}

We denote the isometry group of $(M,g)$ by ${\rm Isom}(M,g)$. When endowed with the $\Cc^0$-topology, this topological group has a compatible Lie group structure where the map
$$
(f,p) \in {\rm Isom}(M,g) \times M \mapsto f(p) \in M
$$
is smooth, see~\cite{MS1939}.

A Riemannian metric $g$ on a smooth manifold $M$ is \emph{real analytic} if $M$ has a real analytic structure compatible with its smooth structure (i.e. a covering by smooth charts where the transition functions are real analytic) such that the metric is real analytic (i.e. real analytic in each of these charts).  In this case, the real analytic structure on $M$ naturally extends to a real analytic structure on the tangent bundle $TM$ and the map
$$
v \in TM \mapsto \exp_{\pi(v)}^g(v) \in M
$$
is real analytic (this follows from the Cauchy--Kovalevskaya theorem).

\subsection{Analytic varieties} A subset $S$ in a real analytic manifold $M$ is a \emph{real analytic variety} if for each $p \in S$ there are an open neighborhood $U$ of $p$ and
finitely many real analytic functions $f_1, \dots, f_m : U \rightarrow \Rb$ such that
$$
U \cap S = \bigcap_{i=1}^m f_i^{-1}(0).
$$

Analytic varieties have the following local Noetherian property, see for instance~\cite[Chapter V Corollary 1]{Narasimhan}.

\begin{theorem}\label{thm:noetherian property}
If $S_1, S_2, \dots$ are closed real analytic varieties in $M$, then for any compact subset $K \subset M$ there exists $N \geq 1$ such that
$$
K \cap \bigcap_{i=1}^\infty S_i = K \cap \bigcap_{i=1}^N S_i.
$$
\end{theorem}

\subsection{Formal power series}\label{section: formal power series} Given a finite dimensional real vector space $V$, let $\Pc(\Rb^d, V)$ denote the vector space of formal power series of the form
$$
\sum_{\alpha} v_\alpha x^\alpha
$$
where $\alpha \in
\mathbb N^d$ is a multi-index, $v_\alpha \in V$, $x \in \Rb^d$, and $x^\alpha = x_1^{\alpha_1} \cdots x_d^{\alpha_d}$. Given a multi-index $\alpha \in \mathbb N^d$, let $\abs{\alpha} : = \alpha_1+ \cdots + \alpha_d$. Then let $\Pc_N(\Rb^d, V) \subset \Pc(\Rb^d, V)$ denote the finite dimensional subspace of polynomials
 of the form
$$
\sum_{\abs{\alpha} \leq N} v_\alpha x^\alpha
$$
and let $\pi_N : \Pc(\Rb^d, V) \rightarrow \Pc_N(\Rb^d, V)$ denote the natural projection given by
$$
\pi_N\left(\sum_{\alpha} v_\alpha x^\alpha  \right) = \sum_{\abs{\alpha} \leq N} v_\alpha x^\alpha.
$$

Given a real analytic function $f : \Oc \rightarrow V$ where $\Oc \subset \Rb^d$ is an open set containing $0$, let
$$
\mathcal{J}(f) \in \Pc(\Rb^d, V)
$$
denote the Taylor series expansion of $f$ at $x=0$.

\section{Local isometries of real analytic Riemannian manifolds}

In this section we consider local isometries of real analytic Riemannian manifolds. We first develop a general result which provides a sufficient condition for a local isometry to extend to a global one. Then we describe when two collections of normal balls in two different manifolds are isometric.

\subsection{Extending local isometries}

It is well known that a local isometry defined on an open connected set in a simply connected real analytic Riemannian manifold extends to a local isometry on the entire manifold (see for instance~\cite[Proposition 11.4]{H2001}). In this section we observe the following consequence of this result.

\begin{proposition}\label{prop:analytic continuation non simply connected case} Suppose $(M,g)$ and $(\hat{M},\hat{g})$ are connected complete real analytic Riemannian manifolds. Assume $U \subset M$ is a connected open set, $u_0 \in U$, and the inclusion map $\iota : U \hookrightarrow M$ induces a surjection $\iota_*: \pi_1(U,u_0) \rightarrow \pi_1(M,u_0)$ of fundamental groups. Then any local isometry $f : U \rightarrow \hat{M}$ extends to a local isometry $F : M \rightarrow \hat{M}$. If, in addition, $f_* : \pi_1(U,u_0) \rightarrow \pi_1(\hat{M},f(u_0))$ is surjective, then $F : M \rightarrow \hat{M}$ is an isometry.
\end{proposition}

\begin{proof} Let $\pi : (\tilde{M}, \tilde{g}) \rightarrow (M,g)$ be the Riemannian universal cover of $M$, that is $\tilde{M}$ is the topological universal cover  with the unique smooth structure making $\pi$ a smooth covering map and $\tilde{g}$ is the unique Riemannian metric that makes $\pi$ a local isometry. Then the real analytic structure on $M$ induces a real analytic structure on $\tilde{M}$ where the metric $\tilde{g}$ is real analytic.

Let $U' \subset \pi^{-1}(U)$ be a connected component. Since $U$ is connected, we have $\pi(U^\prime)= U$. Let $\tilde{f} := f \circ \pi :  U'  \rightarrow \hat{M}$. Then $\tilde{f}$ extends to a real analytic local isometry $\tilde{F} : \tilde{M} \rightarrow \hat{M}$, see for instance~\cite[Proposition 11.4]{H2001}.

Identify $ \pi_1(M,u_0)$ with the deck transformation group of $\tilde{M} \rightarrow M$. We claim that $U'$ is $\pi_1(M,u_0)$-invariant. Fix $\gamma \in \pi_1(M,u_0)$. By assumption $\iota_*: \pi_1(U,u_0) \rightarrow \pi_1(M,u_0)$ is surjective, so there exists a curve $\sigma : \Sb^1 \rightarrow U$ with $\sigma(0)=u_0$ and $[\sigma] = \gamma$,    i.e. $\sigma$ is homotopic to $\gamma$. Then there exists a curve $\tilde{\sigma} : [0,1] \rightarrow \tilde{M}$ such that $\tilde{\sigma}(0) \in U'$ and $\pi \circ \tilde{\sigma}(t) = \sigma(e^{2\pi i t})$ for all $t \in [0,1]$.
By definition of the universal cover, $\gamma \tilde{\sigma}(0) = \tilde{\sigma}(1)$. Further, since $U'$ is connected, we have $\tilde{\sigma}([0,1]) \subset U'$. Hence $\gamma U' \cap U' \neq \emptyset$.
Since $U'$ is a connected component of $\pi^{-1}(U)$, we have $U' = \gamma U'$. Thus $U'$ is $\pi_1(M,u_0)$-invariant.

Next we claim that $\tilde{F}$ is $\pi_1(M,u_0)$-invariant and hence $\tilde{F}$ descends to a local isometry $F : M \rightarrow \hat{M}$. Fix $\gamma \in \pi_1(M,u_0)$. Then for $x \in U'$ we have
$$
\tilde{F}(\gamma x) =\tilde{f}(\gamma x) =  (f \circ \pi)(\gamma x) = (f \circ \pi)(x) = \tilde{f}(x) = \tilde{F}(x).
$$
So $\tilde{F} \circ \gamma = \tilde{F}$ on $U'$. Since $ U'  \subset \tilde{M}$ is open and $\tilde{F}$ is real analytic, we
 have  $\tilde{F} \circ \gamma = \tilde{F}$. Thus  $\tilde{F}$ is $\pi_1(M,u_0)$-invariant.

Now suppose that $f_* : \pi_1(U,u_0) \rightarrow \pi_1(\hat{M},f(u_0))$ is surjective. Since $F$ is a local isometry and $(M,g)$ is complete, $F : M \rightarrow \hat{M}$ is a smooth covering map; see for instance~\cite[Lemma 1.38]{CheegerEbin}.  Since $f = F \circ \iota$, we have
$$
 \pi_1(\hat{M},f(u_0)) = f_*\left( \pi_1(U,u_0)  \right) = F_* \circ \iota_* \left(  \pi_1(U,u_0)  \right) = F_* \left( \pi_1(M, u_0) \right).
 $$
Thus $F_* : \pi_1(M,u_0) \rightarrow \pi_1(\hat{M},u_0)$ is surjective and hence $F$ must be a diffeomorphism. So $F$ is an isometry.
\end{proof}

\subsection{Locally isometric union of  metric balls} In this section we find countably many parameters that describe when two unions of sufficiently small metric balls in real analytic Riemannian manifolds are isometric. This parametrization involves the Taylor series expansion of the metric at the center of each ball and the Taylor series expansion of the transition functions  between normal coordinates on the balls. It is somewhat similar to parametrizations used in some proofs of compactness theorems in Riemannian geometry, see for instance the proof of ~\cite[Theorem 11.3.6]{Petersen}.

Suppose $(M,g)$ is a complete real analytic Riemannian $d$-manifold. Let
$$
\Fc(M) \rightarrow M
$$ be the orthonormal frame bundle of $M$, i.e. the fiber $\Fc_p(M)$ above a point $p \in M$ consists of all ordered orthonormal bases of $(T_pM, g_p)$.

Let $e_1,\dots, e_d$ denote the standard basis of $\Rb^d$ and let $\Sc_d$ denote the vector space of symmetric $d$-by-$d$ real matrices.

Let $\delta : = \injrad_{g}(p)$.  Given
\begin{equation}\label{def-beta}
\beta = (p, (E_1, \dots, E_d)) \in
 \Fc_p(M),
\end{equation}
 let $L_\beta : (\Rb^d, \ip{\cdot, \cdot}) \rightarrow (T_p M, g_p)$ denote the unique linear isometry with $L_\beta(e_i) = E_i$.

 By definition of the exponential map, $\exp^g_p \circ L_\beta$ induces a diffeomorphism $B_{\Rb^d}(0, \delta) \rightarrow B_g(p,\delta)$. Since the metric $g$ is a symmetric tensor  and the tangent bundle of $B_{\Rb^d}(0,\delta)$ is trivializable in a natural way,
   we can view the pullback $(\exp^g_p \circ L_\beta)^* g$ as a real analytic map
$$
(\exp^g_p \circ L_\beta)^* g : B_{\Rb^d}(0, \delta) \rightarrow \Sc_d.
$$
Let
\begin{equation}\label{S_g}
S_g(\beta) : = \mathcal{J}\Big( (\exp^g_p \circ L_\beta)^* g\Big) \in \Pc(\Rb^d, \Sc_d)
\end{equation}
denote the Taylor series expansion of $(\exp^g_p \circ L_\beta)^* g$ at $x=0 \in \Rb^d$.

Next suppose $\beta_1 \in \Fc_{p_1}(M)$, $\beta_2 \in \Fc_{p_2}(M)$, and
$$
\epsilon: =  \injrad_{g}(p_1)-\dist_g(p_1, p_2) >0.
$$
Then
$$
(\exp^g_{p_1} \circ L_{\beta_1})^{-1} \circ (\exp^g_{p_2} \circ L_{\beta_2}) : B_{\Rb^d}(0,\epsilon) \rightarrow \Rb^d
$$
is a well defined real analytic map. Let
\begin{equation}\label{T_g}
T_g(\beta_1,\beta_2) : = \mathcal{J}\Big( (\exp^g_{p_1} \circ L_{\beta_1})^{-1} \circ (\exp^g_{p_2} \circ L_{\beta_2}) \Big) \in \Pc(\Rb^d, \Rb^d)
\end{equation}
denote the Taylor series expansion of $(\exp^g_{p_1} \circ L_{\beta_1})^{-1} \circ (\exp^g_{p_2} \circ L_{\beta_2})$ at $x=0 \in \Rb^d$.

\begin{proposition}\label{prop:local isometry on union of balls}
Suppose $(M,g)$ and $(\hat{M},\hat{g})$ are connected complete real analytic Riemannian manifolds.
Let  $\beta_i \in \Fc_{p_i}(M)$ and $\hat{\beta}_i \in \Fc_{\hat{p}_i}(\hat{M})$ for $i=1,\dots, m$. Fix
$$
0<\delta
<\frac{1}{2} \min_{1\leq j\leq m}\Bigl\{ \injrad_{g}(p_j),  \injrad_{\hat g}(\hat p_j) \Bigr\}
$$
and let
$$
f_i : = \left( \exp^{\hat g}_{\hat{p}_i} \circ L_{\hat{\beta}_i} \right) \circ \left( \exp^g_{p_i} \circ L_{\beta_i} \right)^{-1} : B_g(p_i,\delta) \rightarrow B_{\hat{g}}(\hat{p}_i,\delta).
$$
Assume
\begin{enumerate} \renewcommand{\labelenumi}{$(\alph{enumi})$}
\item $p_1,\dots, p_m$ are pairwise distinct and $\hat{p}_1, \dots, \hat{p}_m$ are pairwise distinct;
\item $S_g(\beta_i)= S_{\hat{g}}(\hat{\beta}_i)$ for $i=1,\dots, m$;
\item $B_g(p_i, \delta) \cap B_g(p_j, \delta) \neq \emptyset$ if and only if $B_{\hat{g}}(\hat{p}_i, \delta) \cap B_{\hat{g}}(\hat{p}_j, \delta) \neq \emptyset$, and in this case
$$
T_g(\beta_i, \beta_j) = T_{\hat{g}}(\hat{\beta}_i, \hat{\beta}_j).
$$
\end{enumerate}
Then
$$
f =\cup_{i=1}^m f_i : \cup_{i=1}^m B_g(p_i, \delta) \rightarrow \cup_{i=1}^m B_{\hat{g}}(\hat{p}_i, \delta)
$$
is a well defined local isometry.
\end{proposition}

\begin{proof} We first show that each $f_i$ is a local isometry. By the definition of injectivity radius, $f_i$ is a well defined real analytic diffeomorphism. Further, the Taylor series expansions at 0 of
$$
(\exp^g_{p_i} \circ L_{\beta_i})^* g \quad \text{and} \quad (\exp^g_{p_i} \circ L_{\beta_i})^*f_i^*\hat{g} = (\exp^{\hat g}_{\hat{p}_i} \circ L_{\hat{\beta}_i})^* \hat{g}
$$
agree. Hence
$$
(\exp^g_{p_i} \circ L_{\beta_i})^* g =  (\exp^g_{p_i} \circ L_{\beta_i})^*f_i^*\hat{g}
$$
on $B_{\Rb^d}(0,\delta)$, which implies that $f_i^*\hat{g} = g$ on $B_g(p_i,\delta)$. Hence $f_i$ is a local isometry.

Next we show that $f =\cup_{i=1}^m f_i$ is well defined. It suffices to fix $1 \leq i, j \leq n$ with
$$
\emptyset \neq B_g(p_i, \delta) \cap B_g(p_j, \delta)
$$
and show that  $f_i = f_j$ on    the intersection.  By the definition of injectivity radius, the map $f_i$ extends to a real analytic diffeomorphism on
$$
B_g(p_i, 2\delta) \supset B_g(p_j, \delta).
$$
Since  $B_g(p_j, \delta)$ is connected, it suffices to show that $f_i = f_j$
in
 a neighborhood of $p_j$. By assumption, the Taylor series expansions at 0 of
\begin{equation*}
 (\exp^g_{p_i} \circ L_{\beta_i})^{-1} \circ (\exp^g_{p_j} \circ L_{\beta_j})  \quad \text{and} \quad  (\exp^{\hat g}_{\hat{p}_i} \circ L_{\hat{\beta}_i})^{-1} \circ (\exp^{\hat g}_{\hat{p}_j} \circ L_{\hat{\beta}_j})
 \end{equation*}
agree. Hence
\begin{equation*}
 (\exp^g_{p_i} \circ L_{\beta_i})^{-1} \circ (\exp^g_{p_j} \circ L_{\beta_j})  =  (\exp^{\hat g}_{\hat{p}_i} \circ L_{\hat{\beta}_i})^{-1} \circ (\exp^{\hat g}_{\hat{p}_j} \circ L_{\hat{\beta}_j})
 \end{equation*}
 on    $B_{\Rb^d}(0,\delta')$ when $\delta'$ is sufficiently small.
  Since $\exp^g_{\hat{p}_i} \circ L_{\hat{\beta}_i}$ induces a diffeomorphism $B_{\Rb^d}(0,\delta) \rightarrow B_{\hat{g}}(\hat{p}_i, \delta)$    that sends the origin to $\hat p_i$, we then have
 $$
 f_i \circ (\exp^g_{p_j} \circ L_{\beta_j}) =  \exp^{\hat g}_{\hat{p}_j} \circ L_{\hat{\beta}_j}
 $$
 on    $B_{\Rb^d}(0,\delta')$. Since $\exp^g_{p_j} \circ L_{\beta_j}$ induces a diffeomorphism $B_{\Rb^d}(0,\delta') \rightarrow B_g(p_j, \delta')$,  we then have
 $$
 f_i = f_j
 $$
 on $B_g(p_j,\delta')$. Therefore, $f_i=f_j$ on $B_g(p_j, \delta)$.
\end{proof}

\section{Proof of Theorem~\ref{thm:continuity implies smoothness in riemannian case in intro}}

In this section we prove Theorem~\ref{thm:continuity implies smoothness in riemannian case in intro}, which we restate here.

\begin{theorem}\label{thm:continuity implies smoothness in riemannian case}
Suppose $M$ and $\hat{M}$ are two connected smooth manifolds. Assume $\{ g_t\}_{t \in \Tc}$ and $\{ \hat{g}_t\}_{t \in \Tc}$ are two smooth families of complete Riemannian metrics on $M$ and $\hat{M}$, respectively. If
\begin{enumerate}[label=$(\alph*)$]
\item there exists a relatively compact connected open set $U \subset M$
such that the inclusion map $\iota : U \hookrightarrow M$ induces   a surjection
  $\iota_*: \pi_1(U,u_0) \rightarrow \pi_1(M,u_0)$ of fundamental groups for some (and hence any) $u_0 \in U$,
\item for each $t \in \Tc$,
\begin{enumerate}[label={$(\roman*)$}]
\item the manifold $M$ has a real analytic structure for which the metric $g_t$ is real analytic,
\item the isometry group $\mathsf{Isom}(M,g_t)$ is discrete, and
\item  there is an isometry $F_t : (M, g_t) \rightarrow (\hat{M}, \hat{g}_t)$,
\end{enumerate}
and
\item the map
$$
(p,t) \in M \times \Tc \longmapsto F_t(p) \in \hat{M}
$$
is continuous,
\end{enumerate}
then the map
$$
(p,t) \in M \times \Tc \longmapsto F_t(p) \in \hat{M}
$$
is smooth.
\end{theorem}

\subsection*{The general strategy} Since an isometry is determined by its 1-jet at a point, it suffices to show that there is some $p \in M$ where the map
$$
t \mapsto (F_t(p), d(F_t)_p)
$$
is smooth. The idea in the proof is to first use Proposition~\ref{prop:local isometry on union of balls} to show that $(F_t(p), d(F_t)_p)$ is determined by a countably infinite collection of constraints
$$
P_n(t, p, F_t(p), d(F_t)_p) = 0, \quad n \in \Nb.
$$
We then use the local Noetherian property, see Theorem~\ref{thm:noetherian property}, to reduce to finitely many constraints and then finally use the implicit function theorem to deduce that there is some $p \in M$ where the map
$$
t \mapsto (F_t(p), d(F_t)_p)
$$
is smooth.

Implementing this strategy is fairly technical and we don't explicitly define the constraints until the very end of the proof, see Lemma~\ref{the n}   for the reduction to finitely many constraints and Lemma~\ref{lem:explicit constraints} to find the functional equation for  the implicit function theorem.

\subsection*{Notations and reductions}
To prove Theorem~\ref{thm:continuity implies smoothness in riemannian case} it suffices to fix $t_0 \in \Tc$ and prove that $(p,t) \mapsto F_t(p)$ is smooth in a neighborhood of $M \times \{t_0\}$.

We will use  the letters $p,q$ to denote points in $M$ and the letters $x,y,z$ to denote points in $\hat{M}$.

\subsection*{Step 1: Using Proposition~\ref{prop:local isometry on union of balls}}

Let $V:= F_{t_0}(U)$. Since $F_{t_0} : M \rightarrow \hat{M}$ is a diffeomorphism, $V$ is an open connected relatively compact subset of $\hat{M}$ and the inclusion map $\iota : V \hookrightarrow \hat{M}$ induces a surjection
$\iota_*: \pi_1(V,v_0) \rightarrow \pi_1(\hat{M},v_0)$ for any $v_0 \in V$.

\begin{lemma} There exists some $x \in V$ such that
$$
\sigma(x) \neq x
$$
for all non-trivial $\sigma  \in \mathsf{Isom}(\hat{M},\hat{g}_{t_0})$.
\end{lemma}

\begin{proof}  Recall that an isometry of a connected Riemannian manifold which is the identity on an open set is the identity everywhere. This follows from Equation~\eqref{efp}, which implies that for an isometry $f : (N, h) \rightarrow (N,h)$ the set
$$
\{ x \in N : f(x) = x \text{ and } df_x = \id_{T_x N} \}
$$
is open and closed.

Then for every non-trivial $\sigma \in \mathsf{Isom}(\hat{M},  \hat g_{t_0}
)$, the set
$$
{\rm Fix}(\sigma ) : = \{ x \in \hat{M} : \sigma (x) = x\}
$$
is closed and has empty interior. Since $\mathsf{Isom}(\hat{M},\hat{g}_{t_0})$ is discrete, it is countable. So by the Baire category theorem there exists some
\begin{equation*}
x \in  V \setminus \bigcup \left\{ {\rm Fix}(\sigma ) : \sigma  \in \mathsf{Isom}(\hat{M},
\hat g_{t_0}
), \, \sigma  \neq \id_{\hat M}\right\}.  \qedhere
\end{equation*}
\end{proof}

Now fix $x_1 \in V$ such that
$$
\sigma (x_1) \neq x_1
$$
for all non-trivial $\sigma  \in \mathsf{Isom}(\hat{M},\hat{g}_{t_0})$. Since $\mathsf{Isom}(\hat{M},\hat{g}_{t_0})$
  acts properly on $\hat{M}$ and  is discrete,
 the orbit
$$
\mathsf{Isom}(\hat{M},\hat{g}_{t_0})(x_1) \subset \hat{M}
$$
is discrete and hence
\begin{equation}\label{eqn:definition of r0}
r_0:=\inf\left\{ \dist_{\hat{g}_{t_0}}(x_1, \sigma (x_1)) :\sigma  \in \mathsf{Isom}(\hat{M},\hat{g}_{t_0}), \, \sigma  \neq \id_{\hat M}\right\} \in [0,+\infty) \cup \{+\infty\}
\end{equation}
is    positive  (we define the infimum of the empty set to be $+\infty$).

Given $A \subset \hat{M}$, $r > 0$, and $t \in \Tc$ let
$$
\Nc_t(A, r) : = \left\{ x \in \hat{M} : \dist_{\hat{g}_t}(x, A) < r \right\}.
$$
Since $V$ is relatively compact and the map $x \in \hat M \mapsto \injrad_{\hat g_{t_0}}(x)$ is continuous (see~\cite[Proposition 10.37]{LeeRiemGeom}), we can fix $0 < \delta_0 < r_0$ such that
\begin{equation}\label{eqn:lower bound on inj rad}
\injrad_{\hat{g}_{t_0}}(  x) >
    4\delta_0
   \end{equation}
   for all $x \in \Nc_{t_0}(\overline{V}, \delta_0)$.

Next fix $x_2,\dots, x_m \in \overline{V}$ such that
\begin{equation}\label{eqn:pi balls cover}
\overline{V} \subset \bigcup_{i=1}^m B_{\hat g_{t_0}}(x_i,\delta_0).
\end{equation}

\begin{lemma}\label{lem:intersection}
 There exist $\delta \in (0,
 \delta_0
 )$ and $
   r \in (0, \delta_0/2)$ such that: if $y_i \in  B_{\hat g_{t_0}}(x_i, r)$ for $i=1,\dots, m$, then
$$
B_{\hat g_{t_0}}(y_i,\delta)  \cap B_{\hat g_{t_0}}(y_j,\delta) \neq \emptyset \Longleftrightarrow {B}_{\hat g_{t_0}}(x_i,\delta) \cap B_{\hat g_{t_0}}(x_j,\delta) \neq \emptyset
$$
and
$$
 \overline{V} \subset \bigcup_{i=1}^m B_{\hat g_{t_0}}(y_i,\delta).
$$
 \end{lemma}

 \begin{proof}    Fix $1 \leq i,j \leq m$. When the open set $B_{\hat g_{t_0}}(x_i,\delta)  \cap B_{\hat g_{t_0}}(x_j,\delta)$ is not empty, the open set $B_{\hat g_{t_0}}(y_i,\delta)  \cap B_{\hat g_{t_0}}(y_j,\delta)$ is still non-empty when
 $\delta\in(0,\delta_0)$ is sufficiently close to $\delta_0$
 and $r>0$ is sufficiently close to zero, and vice versa. Since $m$ is finite, we can find such $\delta,r$ satisfying the first assertion. We can also fix $\delta, r$ to satisfy the second assertion since $\overline V$ is compact and $\{B_{\hat g_{t_0}}(x_i,\delta_0)\}$ is a finite and open covering.
\end{proof}

For $t \in \Tc$, let
$$
\Fc_t(\hat{M}) \rightarrow \hat{M}
$$
be the orthonormal frame bundle with respect to the metric $\hat{g}_t$. Then for $i=1,\dots, m$, we consider    the fiber bundle
$$
\Oc_{t,i}: = \Fc_t(\hat{M})|_{B_{\hat g_{t_0}}(x_i,r)}
$$
 above $B_{\hat g_{t_0}}(x_i,r)$. Also let
$$
E: = \left\{ (i,j) : 1 \leq i < j \leq m \text{ and } B_{\hat g_{t_0}}(x_i,\delta) \cap B_{\hat g_{t_0}}(x_j,\delta) \neq \emptyset\right\}.
$$

Now fix a sufficiently small open set $\Tc_0 \subset \Tc$ containing $t_0$. Using the definitions of $S_g, T_g$ in ~\eqref{S_g} and ~\eqref{T_g},   we claim that we can
 define a map
$$
\Phi_t : \Oc_{t,1} \times \cdots \times \Oc_{t,m} \rightarrow \Pc(\Rb^d, \Sc_d)^m \times \Pc(\Rb^d, \Rb^d)^{\abs{E}}
$$
for $t \in \Tc_0$
by
$$
\Phi_t\left(\beta_1,\dots, \beta_m\right) = \big( \left( S_{\hat{g}_t}(\beta_1), \dots, S_{\hat{g}_t}(\beta_m)\right), \left( T_{\hat{g}_t}(\beta_i, \beta_j) : (i,j) \in E \right) \big).
$$

\begin{lemma}\label{lem:Phit well defined} After possibly shrinking $\Tc_0$, $\Phi_t$ is well-defined when $t \in \Tc_0$. \end{lemma}

\begin{proof}
 It is clear that each $S_{\hat{g}_t}(\beta_j)$ is well-defined. To check that $T_{\hat{g}_t}(\beta_i,\beta_j)$ can be defined by \eqref{T_g},  we need to show that if $(i,j) \in E$, $\beta_i \in \Oc_{t,i}$ has base point $y_i \in B_{\hat g_{t_0}}(x_i,r)$, and $\beta_j  \in \Oc_{t,j}$ has base point $y_j \in B_{\hat g_{t_0}}(x_j,r)$, then
$$
\dist_{\hat{g}_t}(y_i, y_j) < \injrad_{\hat{g}_t}(y_i).
$$
In this case, Equation~\eqref{eqn:lower bound on inj rad} implies that
$$
 \injrad_{\hat{g}_{t_0}}(y_i) \geq 4\delta_0
 $$
and
 $$
 \dist_{\hat{g}_{t_0}}(y_i, y_j) \leq \dist_{\hat{g}_{t_0}}(x_i,x_j) + 2r \leq 2\delta + 2r \leq 3 \delta_0
 $$
 by  the bounds on $\delta, r$ in Lemma~\ref{lem:intersection}. Further, the map
 $$
(y,t) \in \hat M \times \Tc  \mapsto  \injrad_{\hat{g}_{t}}(y) \in \Rb
 $$
 is lower semicontinuous (see~\cite[Remark on pg. 170]{Ehrlich}) and the map
 $$
 (y,y^\prime,t) \in \hat{M} \times \hat{M} \times \Tc \longmapsto \dist_{\hat{g}_t}(y, y^\prime) \in \Rb
 $$
 is continuous. So after possibly shrinking $\Tc_0$, $\Phi_t$ is well-defined when $t \in \Tc_0$.
\end{proof}

\begin{lemma}\label{lem:total map is injective} $\Phi_{t_0}$ is injective. \end{lemma}

\begin{proof} Suppose
$$
\Phi_{t_0}\left(\beta_1,\dots, \beta_m\right) = \Phi_{t_0}\left(\beta_1^\prime,\dots, \beta_m^\prime\right).
$$
For each $i=1,\dots,m$, let $y_i,y_i' \in B_{\hat g_{t_0}}(x_i,r)$ be the base points of $\beta_i, \beta_i'$ respectively. We apply Proposition~\ref{prop:local isometry on union of balls}  to
$\{ \beta_i\}$, $\{\beta_i^\prime\}$. By Equation~\eqref{eqn:lower bound on inj rad}
\begin{align*}
\frac{1}{2}\min_{1 \leq i \leq m} \left\{ \injrad_{\hat{g}_{t_0}}(y_i), \injrad_{\hat{g}_{t_0}}(y_i^\prime)\right\} \geq 2 \delta_0 > 2 \delta > \delta.
\end{align*}
Further, by Lemma~\ref{lem:intersection} we have
 $$
 \overline V\subset\bigcup_{i=1}^m B_{\hat g_{t_0}}(y_i,\delta).
 $$
 Hence by Proposition~\ref{prop:local isometry on union of balls}  there exists a local isometry $f : V \rightarrow \hat{M}$ such that $df(\beta_i) = \beta_i^\prime$ for $i=1,\dots,m$.

By Proposition~\ref{prop:analytic continuation non simply connected case}, $f$ extends to an isometry $F : (\hat{M},\hat{g}_{t_0}) \rightarrow (\hat{M}, \hat{g}_{t_0})$. Since $dF(\beta_1) = \beta_1^\prime$, we have $F(y_1)=y_1'$. Thus
\begin{align*}
\dist_{\hat{g}_{t_0}}(F(x_1), x_1)& \leq \dist_{\hat{g}_{t_0}}(F(x_1),F(y_1))+\dist_{\hat{g}_{t_0}}(F(y_1),y_1^\prime)+\dist_{g_{t_0}}(y_1^\prime,x_1)\\
&\leq r + 0 + r <  \delta_0 < r_0.
\end{align*}
So by Equation~\eqref{eqn:definition of r0}, $F = \id_{\hat{M}}$. Thus
\begin{equation*}
\left(\beta_1,\dots, \beta_m\right) =\left(\beta_1^\prime,\dots, \beta_m^\prime\right). \qedhere
\end{equation*}
\end{proof}

\subsection*{Step 2: Using Theorem~\ref{thm:noetherian property}}

Using the notation introduced in Section~\ref{section: formal power series}, for $n \in \Nb$ and $t \in \Tc_0$, let
$$
\Phi_t^n : \Oc_{t,1} \times \cdots \times \Oc_{t,m} \rightarrow \Pc_n(\Rb^d, \Sc_d)^m \times \Pc_n(\Rb^d, \Rb^d)^{\abs{E}}
$$
be $\Phi_t$ post-composed with the natural projections. Notice that $\Pc_n(\Rb^d, \Sc_d)^m \times \Pc_n(\Rb^d, \Rb^d)^{\abs{E}}$ is a finite dimensional vector space and each $\Phi_t^n$ is a real analytic map.

For $i =1,\dots, m$, fix a compact subset
$$
K_i \subset \Oc_{t_0,i} \subset \Fc_{t_0}(\hat{M})
$$
 with non-empty interior.

\begin{lemma}\label{the n}
 For $n \in \Nb$ sufficiently large, $\Phi_{t_0}^n$ is injective on $K_1 \times \cdots \times K_m$. \end{lemma}

\begin{proof}     For every $n$, define  the difference function
$$
D_n :  \left(\Oc_{t_0,1} \times \cdots \times \Oc_{t_0,m}\right)^2 \rightarrow \Pc_n(\Rb^d, \Sc_d)^m \times \Pc_n(\Rb^d, \Rb^d)^{\abs{E}}
$$
by
$$
D_n\left( \beta_1,\dots, \beta_m, \beta_1^\prime,\dots, \beta_m^\prime\right)= \Phi_{t_0}^n\left(\beta_1,\dots, \beta_m\right) - \Phi_{t_0}^n\left(\beta_1^\prime,\dots, \beta_m^\prime\right).
$$
Let
$$
\Delta  \subset \left(\Oc_{t_0,1} \times \cdots \times \Oc_{t_0,m}\right)^2
$$
denote the diagonal. By Lemma~\ref{lem:total map is injective},
$$
\Delta = \bigcap_{n =1}^\infty D_n^{-1}(0).
$$
So by Theorem~\ref{thm:noetherian property} there exists $N \geq 1$ such that
$$
\left(K_1 \times \cdots \times K_m\right)^2 \cap \Delta = \left(K_1 \times \cdots \times K_m\right)^2 \cap\bigcap_{n =1}^N D_n^{-1}(0).
$$
Then $\Phi_{t_0}^n$ is injective on $K_1 \times \cdots \times K_m$ for all $n \geq N$.
\end{proof}

\subsection*{Step 3: Using the constant rank theorem} Fix $n$ sufficiently large so that $\Phi_{t_0}^n$ is injective on $K_1 \times \cdots \times K_m$. In this step we use the constant rank theorem to deduce that the derivative of the map $\Phi_{t_0}^n$ is injective at some point in $K_1 \times \cdots \times K_m$.

\begin{lemma} There exists some $B_0 \in
{\rm int}(K_1 \times \cdots \times K_m)$ such that the derivative $d(\Phi_{t_0}^n)_{B_0}$ is injective. \end{lemma}

\begin{proof} Fix $B_0 \in {\rm int}( K_1 \times \cdots \times K_m)$ such that $d(\Phi_{t_0}^n)_{B_0}$ has maximal rank
   amongst all interior points. Since the rank of $d(\Phi_{t_0}^n)_{B}$ is lower semicontinuous, there exists a neighborhood of $B_0$ where $d(\Phi_{t_0}^n)_{B}$ has constant rank. By the constant rank theorem, see~\cite[Theorem 4.12]{LeeSmthMflds},
   there  exist  local coordinates centered at $B_0$ and local coordinates centered at $\Phi_{t_0}^n(B_0)$ such that the map $\Phi_{t_0}^n$ has the form
$$
(x_1, \dots, x_{D_1}) \mapsto (x_1, \dots, x_r, \underbrace{0, \dots, 0}_{(D_2-r)-\text{ times}})
$$
with $D_1 := m \dim \Fc_{t_0}(\hat{M})$, $D_2 := \dim( \Pc_n(\Rb^d, \Sc_d)^m \times \Pc_n(\Rb^d, \Rb^d)^{\abs{E}})$, and $r := {\rm rank} \, d(\Phi_{t_0}^n)_{B_0}$. Since $\Phi_{t_0}^n$ is injective on $K_1 \times \cdots\times K_m$, we must have $r = D_1$. Hence $d(\Phi_{t_0}^n)_{B_0}$ is injective.
\end{proof}

\subsection*{Step 4: Using the implicit function theorem} In this step we use a consequence of the implicit function theorem:  immersions are local embeddings.

Since $(\hat{g}_t)_{t \in \Tc}$ is a smooth family,
$$
\bigcup_{t \in \Tc} \{ t\} \times \Fc_t(\hat{M})^m
$$
is a smooth submanifold of $\Tc_0 \times (\oplus_{i=1}^d T\hat{M})^m$. Recall that $\Oc_{t,i}=\mathcal F_t(\hat{M})|_{
B_{\hat g_{t_0}}
(x_i,r)}$ is an open subset of $\Fc_t(\hat{M})$ and hence
\begin{equation}\label{eqn:defn of E tilde}
   \widetilde E : = \bigcup_{t \in \Tc_0} \{t\} \times \left( \Oc_{t,1} \times \cdots \times \Oc_{t,m}\right)
\end{equation}
is an open subset of $\cup_{t \in \Tc} \{ t\} \times \Fc_t(\hat{M})^m$. Further, the map
\begin{align*}
G : \widetilde E \rightarrow \Tc_0 \times \Pc_n(\Rb^d, \Sc_d)^m \times \Pc_n(\Rb^d, \Rb^d)^{\abs{E}}
\end{align*}
   defined by
$$
G(t,B) = (t,\Phi_t^n(B)),
$$
is smooth.  Since the rank of $d(\Phi_{t}^n)_{B}$ is lower semicontinuous,
 there exists a neighborhood $\mathcal{D}$ of $(t_0, B_0)$ in $\widetilde E$ such that for all $(t,B) \in \mathcal{D}$ the derivative $d(\Phi_{t}^n)_{B}$    is injective. Since
$$
dG_{(t,B)} = \begin{pmatrix} 1 & 0 \\ \ast &  d(\Phi_{t}^n)_{B} \end{pmatrix},
$$
then $d(G)_{(t,B)}$    is injective for all $(t,B) \in \mathcal{D}$. Since every immersion is a local embedding, see~\cite[Theorem 4.25]{LeeSmthMflds}, after shrinking $\mathcal{D}$, we can assume that
$$
G(\mathcal{D}) \subset \Tc_0 \times \Pc_n(\Rb^d, \Sc_d)^m \times \Pc_n(\Rb^d, \Rb^d)^{\abs{E}}
$$
is an embedded submanifold and $G|_{\mathcal{D}} : \mathcal{D} \rightarrow G(\mathcal{D})$ is a diffeomorphism. Therefore, $G|_{\mathcal D}^{-1}\colon G(\mathcal D)\to\mathcal D$ exists and hence it is smooth.

To be able to use the inverse, we need some preparation.

\subsection*{Step 5: The 1-jet is smooth at some
 point}
  Let $$
\Fc_t(M) \rightarrow M
$$
be the orthonormal frame bundle with respect to the metric $g_t$ and let $\Fc_t(p) := \Fc_t(M)|_{p}$ denote the fiber above a point $p \in M$.

Suppose $B_0 = (\beta_{01},\dots, \beta_{0m})$ is as above and $\beta_{0i}$ has base point $z_i \in \hat{M}$. As in \eqref{def-beta}, we write
$$
\beta_{0i}=\big(z_i, (E_{i,1}, \dots, E_{i,d})\big).
$$
 Let $p_i := F_{t_0}^{-1}(z_i) \in M$ and
 $$
\alpha_{0i} :=
\big(p_i,\left(d(F_{t_0})_{p_i}^{-1}E_{i,1}, \dots, d(F_{t_0})_{p_i}^{-1}E_{i,d}\right)\big) \in
 \Fc_{t_0}(p_i).
$$

Next fix an open set $\Tc_1 \subset \Tc_0$ containing $t_0$ and a smooth map $\gamma$ defined on $\Tc_1$ such that $\gamma(t_0) = (\alpha_{01},\dots, \alpha_{0m})$ and
\begin{equation}\label{def-gamma(t)}
\gamma(t)= (\gamma_1(t), \dots, \gamma_m(t)) \in \Fc_t(p_1) \times \cdots \times \Fc_t(p_m)
\end{equation}
for all $t \in \Tc_1$.

  Using the continuity of $t \mapsto F_t$, we observe the following.

\begin{lemma}\label{compactness+} The map
$$
t\in \mathcal T_1 \longmapsto dF_t\gamma(t)\in \Fc_t(\hat{M})^m
$$
is continuous. \end{lemma}

\begin{remark} Here and in what follows $dF_{t}\gamma(t)=(dF_{t}\gamma_1(t),\dots, dF_{t}\gamma_m(t))$. \end{remark}

\begin{proof} It suffices to fix $1 \leq i \leq m$ and show that $t \mapsto d(F_t)_{p_i}$ is continuous.  Notice that
$$
F_t = {\exp}^{\hat g_t}_{F_t(p_i)} \circ d(F_t)_{p_i} \circ \left(\exp^{ g_t}_{p_i}\right)^{-1}
$$
near $p_i$. So
$$
d(F_t)_{p_i} =\left( {\exp}^{\hat g_t}_{F_t(p_i)}\right)^{-1} \circ F_t \circ \exp^{g_t}_{p_i}
$$
near $0 \in T_{p_i} M$. Since $t \mapsto F_t$ is continuous, Observation~\ref{obs:smoothness of exponential maps} implies that $t \mapsto d(F_t)_{p_i}$ is continuous.
\end{proof}

Since
$dF_{t_0}\gamma(t_0) = B_0$
 and $t \mapsto dF_t\gamma(t)$ is continuous, by shrinking $\Tc_1$ we may also assume that
$$
(t, dF_t\gamma(t)) \in \mathcal{D}
$$
for all $t \in \Tc_1$.

Define $H : \Tc_1 \rightarrow \Tc_1 \times \Pc_n(\Rb^d, \Sc_d)^m \times \Pc_n(\Rb^d, \Rb^d)^{\abs{E}}$ by
$$
H(t) =  \Big( t, \big( \pi_nS_{g_t}(\gamma_1(t)), \dots, \pi_nS_{g_t}(\gamma_m(t))\big), \big( \pi_nT_{g_t}(\gamma_i(t), \gamma_j(t)) : (i,j) \in E \big) \Big)
$$
where we abuse notation and use $\pi_n$ to denote both projections
$$
 \Pc(\Rb^d, \Sc_d) \rightarrow  \Pc_n(\Rb^d, \Sc_d) \quad \text{and} \quad \Pc(\Rb^d, \Rb^d) \rightarrow \Pc_n(\Rb^d, \Rb^d).
$$

\begin{lemma}\label{lem:explicit constraints}
After possibly shrinking $\Tc_1$, $H$ is well defined, smooth, and
\begin{equation*}
H(t) = G\left(t, dF_t\gamma(t)\right)
\end{equation*}
for all $t \in \Tc_1$.
\end{lemma}

\begin{proof}  It is clear that each $S_{g_t}(\gamma_i(t))$ is well defined. To check that $T_{g_t}(\gamma_i(t), \gamma_j(t))$ can be defined by \eqref{T_g},  we need to show that if $(i,j) \in E$, then
$$
\injrad_{g_t}(p_i) - \dist_{g_t}(p_i,p_j) > 0
$$
(recall that $\gamma_i(t)$ has constant base point $p_i$). At $t_0$, we have
$$
\injrad_{g_{t_0}}(p_i) - \dist_{g_{t_0}}(p_i,p_j)  = \injrad_{\hat{g}_{t_0}}(z_i) - \dist_{\hat{g}_{t_0}}(z_i,z_j) > 0
$$
since $F_{t_0}$ is an isometry mapping $p_1,\dots, p_m$ to $z_1,\dots, z_m$. Then arguing as in the proof of Lemma~\ref{lem:Phit well defined}, we can shrink $\Tc_1$ so that
$$
\injrad_{g_t}(p_i) - \dist_{g_t}(p_i,p_j) > 0
$$
for all $t \in \Tc_1$ and $(i,j) \in E$.

It follows from the definitions that $H$ is smooth. For the final assertion, note that
\begin{align*}
G  (& t,  dF_t\gamma(t))  =(t,\Phi^n_t(dF_t\gamma(t)))\\
&= \Bigl(t,\bigl(\pi_n \left( S_{\hat{g}_t}(dF_t\gamma_j(t))\right)  : 1\leq j\leq m\bigr) ,    \bigl( \pi_n\left( T_{\hat{g}_t}\left(dF_t\gamma_i(t), dF_t\gamma_j(t)\right)\right) : (i,j) \in E \bigr) \Bigr).
\end{align*}
Since $F_t$ is an isometry, we have
$$
{\exp}^{\hat g_t}_{F_t(p_i)} \circ d(F_t)_{p_i} = F_t \circ \exp^{ g_t}_{p_i}.
$$
So, by definition,
$$
S_{\hat{g}_t}(dF_t\gamma_j(t)) = S_{g_t}( \gamma_j(t)) \quad \text{and} \quad  T_{\hat{g}_t}\left(dF_t\gamma_i(t), dF_t\gamma_j(t))=T_{g_t}(\gamma_i(t), \gamma_j(t)\right).
$$
Thus $H(t) = G\left(t, dF_t\gamma(t)\right)$.
\end{proof}

\begin{lemma}\label{lem:one jet is smooth at a point} The map
$$
t \in \Tc_1 \mapsto (F_t(p_1), d(F_t)_{p_1})
$$
is smooth.
\end{lemma}

\begin{proof}   Recall that $\widetilde E$ was defined in Equation~\eqref{eqn:defn of E tilde}. Let $\rho : \widetilde E \rightarrow \Fc_t(\hat{M})$ denote the projection
$$
\rho(t, \beta_1,\dots, \beta_m) = \beta_1.
$$
Then  the map
$$
 t \in \mathcal T_1 \longmapsto  dF_t \gamma_1(t) =\left(\rho \circ (G|_{\mathcal{D}})^{-1} \circ H\right)(t)
$$
is smooth. Since $\gamma_1(t) \in \Fc_t(p_1)$ for all $t \in \Tc_1$, this implies that
$$
t \in \Tc_1 \mapsto (F_t(p_1), d(F_t)_{p_1})
$$
is smooth.
\end{proof}

\subsection*{Step 6: Finishing the proof} Using the exponential map and Lemma~\ref{lem:one jet is smooth at a point}, we finish the proof.

\begin{lemma} The map
$$
(p,t) \in M\times  \Tc_1 \mapsto F_t(p)
$$
is smooth.
\end{lemma}

\begin{proof}  It suffices to fix  $q_0 \in M$ and prove that the map is smooth in a neighborhood of $\{q_0\} \times \Tc_1$.      Fix a smooth path $\sigma : [0,1] \rightarrow M$ with $\sigma(0)= p_1$ and $\sigma(1)=q_0$.   Since the map
$$
(p,t) \in M \times \Tc \mapsto \injrad_{g_{t}}(p)
$$
is lower semi-continuous, (see~\cite[Remark on pg. 170]{Ehrlich}), there exists $r > 0$ such that
$$
\injrad_{g_t}(\sigma(s)) > r
$$
for every $s \in [0,1]$ and $t \in \overline{\Tc_1}$.

 By our choice of $r>0$, when $t \in \overline{\Tc_1}$ and $s \in [0,1]$,
\begin{equation}\label{eqn:local equation for Ft}
F_t =  {\exp}^{\hat g_t}_{F_t(\sigma(s))}  \circ d(F_t)_{\sigma(s)} \circ \left(\exp^{ g_t}_{\sigma(s)}\right)^{-1}
\end{equation}
on $B_{g_t}(\sigma(s), r)$.

Next fix $0 = s_0 < \dots < s_k =1$ such that
$$
B_{g_t}(\sigma(s_i), r/2) \cap B_{g_t}(\sigma(s_{i+1}), r/2) \neq \emptyset
$$
for all $t \in \overline{\Tc_1}$ and $i =0,\dots, k-1$. Then by Equation~\eqref{eqn:local equation for Ft},
$$
F_t(\sigma(s_{i+1})) = {\exp}^{\hat g_t}_{F_t(\sigma(s_i))} \circ d(F_t)_{\sigma(s_{i})} \circ \left(\exp^{ g_t}_{\sigma(s_i)}\right)^{-1}(\sigma(s_{i+1}))
$$
and
$$
d(F_t)_{\sigma(s_{i+1})} = d\left( {\exp}^{\hat g_t}_{F_t(\sigma(s_i))} \circ d(F_t)_{\sigma(s_{i})} \circ \left(\exp^{ g_t}_{\sigma(s_i)}\right)^{-1}\right)_{\sigma(s_{i+1})}
$$
for all $t \in \Tc_1$ and $i =0,\dots, k-1$. Since the map
$$
t \mapsto \left(F_t(p_1), d(F_t)_{p_1}\right)=\left(F_t(\sigma(s_0)), d(F_t)_{\sigma(s_0)}\right)
$$
is smooth on $\Tc_1$, by induction and Observation~\ref{obs:smoothness of exponential maps} the map
$$
t \mapsto \left(F_t(q_0), d(F_t)_{q_0}\right)=\left(F_t(\sigma(s_k)), d(F_t)_{\sigma(s_k)}\right)
$$
is smooth on $\Tc_1$. Then by Equation~\eqref{eqn:local equation for Ft}, the map
$$
(q,t) \in \bigcup_{t \in \Tc_1}  B_{g_t}(q_0, r)\times \{t\} \mapsto F_t(q)
$$
is smooth. Since $q_0 \in M$ was arbitrary, this completes the proof.
\end{proof}

\section{Smoothness of families of isometries between compact manifolds}

In this section we use Theorem~\ref{thm:continuity implies smoothness in riemannian case} to prove Theorem~\ref{thm:riemannian cpct case}, which we restate here.

\begin{theorem}\label{thm:riemannian cpct case in paper}
Suppose $M$ and $\hat{M}$ are two compact smooth manifolds. Assume $\{ g_t\}_{t \in \Tc}$ and $\{ \hat{g}_t\}_{t \in \Tc}$ are two smooth families of Riemannian metrics on $M$ and $\hat{M}$, respectively. If for each $t \in \Tc$,
\begin{enumerate}\renewcommand{\labelenumi}{$(\alph{enumi})$}
\item the isometry group $\mathsf{Isom}(M,g_t)$ is trivial, and
\item there is an isometry $F_t : (M, g_t) \rightarrow (\hat{M}, \hat{g}_t)$,
\end{enumerate}
then the map
$$
(p,t) \in M \times \Tc \longmapsto F_t(p) \in \hat{M}
$$
is smooth.
\end{theorem}

The idea is to apply the Ricci flow to each metric to obtain real analytic metrics and then to apply Theorem~\ref{thm:continuity implies smoothness in riemannian case}.

\begin{proof}
Given a compact smooth Riemannian manifold $(N,h)$, let $h^{(s)}$, with $s \in [0,s_0)$, denote the solution to the Ricci flow starting at $h$, that is
$$
\frac{\partial}{\partial s} h^{(s)} = -{\rm Ric}(h^{(s)}), \quad h^{(0)} = h.
$$
By~\cite{real_analyticity_Ricci_flow}, for each $s \in (0,s_0)$ the manifold $N$ has a real analytic structure where the metric $h^{(s)}$ is real analytic.

Fix $t_0 \in \Tc$. Using Hamilton's proof of the short time existence of the Ricci flow we obtain the following.

\begin{lemma}[{consequence of the proof of~\cite[Theorem 4.2]{HamiltonRicciFlow}}]\label{lem:consequence of Hamiltons Ricci flow} There are a neighborhood $\Tc_0$ of $t_0$ in $\Tc$ and $s_1> 0$ such that
the map
$$(s,t)\in[0,s_1]\times\Tc_0\longmapsto \left(g_t^{(s)},  \hat{g}_t^{(s)}\right)$$
is smooth.
\end{lemma}

\begin{proof} In the proof of Theorem 4.2 in~\cite{HamiltonRicciFlow}, see page 263 in particular, Hamilton establishes short time existence of the Ricci flow by using the version of the Nash-Moser inverse function theorem stated in~\cite[pg. 171]{Hamilton2}. The definition of the map Hamilton constructs  and the conclusion of the inverse function theorem immediately imply the smoothness of the mapping.
\end{proof}

For $t \in \Tc_0$, let $h_t:= g_t^{(s_1)}$ and $\hat{h}_t:=\hat{g}_t^{(s_1)}$. By Lemma~\ref{lem:consequence of Hamiltons Ricci flow}, we have:
\begin{enumerate}\renewcommand{\labelenumi}{(\roman{enumi})}
\item For all $t \in \Tc_0$, the families of metrics $\{ h_t\}_{t \in \Tc_0}$ and $\{ \hat h_t\}_{t \in \Tc_0}$ are smooth.
\end{enumerate}
Kotschwar's backward uniqueness theorem for the Ricci flow~\cite{backward_uniqueness} implies that:
\begin{enumerate}\renewcommand{\labelenumi}{(\roman{enumi})}\setcounter{enumi}{1}
\item For all $t \in \Tc_0$, the isometry groups $\mathsf{Isom}(M,h_t)$ and $\mathsf{Isom}(\hat{M},\hat{h}_t)$ are trivial.
\end{enumerate}
Also, notice that $s \mapsto F_t^* \hat{g}_t^{(s)}$ solves the Ricci equation starting at $g_t$. So by uniqueness of solutions to the Ricci flow:
\begin{enumerate}\renewcommand{\labelenumi}{(\roman{enumi})}\setcounter{enumi}{2}
\item For all $t \in \Tc_0$, the map $F_t$ is an isometry $(M, h_t) \rightarrow (\hat{M}, \hat{h}_t)$.
\end{enumerate}
So to apply Theorem~\ref{thm:continuity implies smoothness in riemannian case}, we only have to show that $F_t$ is continuous in $t$.

\begin{lemma} The map $(p,t) \in M \times \Tc \mapsto F_t(p) \in \hat{M}$ is continuous. \end{lemma}

\begin{proof} By compactness  it suffices to verify that if $(p_n, t_n) \rightarrow (p_\infty,t_\infty)$ and $F_{t_n}(p_n) \rightarrow x$ as $n$ tends to $\infty$, then   $x = F_{t_\infty}(p_\infty)$.

 Since $F_{t_n}$ is an isometry from $(M,g_{t_n})$ to $(\hat M,\hat g_{t_n})$ and  the families of metrics are smooth, there exists $C > 0$ (independent of $n$) such that each $F_{t_n}$ is $C$-bi-Lipschitz with respect to $g_{t_\infty}$ and $\hat g_{t_\infty}$.     Then, using the Arzel\'a-Ascoli theorem, there exists a subsequence   $F_{t_{n_j}}$ converging to a homeomorphism $f : M \rightarrow \hat M$.
Further,
$$
\dist_{\hat{g}_{t_\infty}}(f(p), f(q)) = \lim_{j \rightarrow \infty} \dist_{\hat g_{t_{n_j}}}(F_{t_{n_j}}(p), F_{t_{n_j}}(q)) = \lim_{j \rightarrow \infty} \dist_{g_{t_{n_j}}}(p,q) = \dist_{g_{t_\infty}}(p,q)
$$
for all $p,q \in M$. So by the Myers--Steenrod theorem, $f$ is a smooth diffeomorphism and $f^* \hat{g}_{t_\infty} = g_{t_\infty}$. Thus $f : (M, g_{t_\infty}) \rightarrow (\hat M, \hat g_{t_\infty})$ is an isometry. Since $\mathsf{Isom}(M,g_{t_\infty})$ and $\mathsf{Isom}(\hat{M},\hat{g}_{t_\infty})$ are trivial, this is only possible if $f = F_{t_\infty}$. Thus
$$
x = f(p) = F_{t_\infty}(p)
$$
and the proof is complete.
\end{proof}

Then Theorem~\ref{thm:continuity implies smoothness in riemannian case} implies that the map
$$
(p,t) \in M \times \Tc \longmapsto F_t(p) \in \hat{M}
$$
is smooth.
\end{proof}

\section{Smoothness of families of biholomorphisms}

In this section we use Theorem~\ref{thm:continuity implies smoothness in riemannian case} to prove Theorem~\ref{thm:smoothness of family of biholomorphisms}, which we restate here.

\begin{theorem}\label{thm:smoothness of family of biholomorphisms in paper}
Suppose $\{\Omega_t\}_{t \in \Tc}$ and $\{\hat{\Omega}_t\}_{t \in \Tc}$ are two smooth families of strongly pseudoconvex domains in $\Cb^d$. If for each $t \in \Tc$
\begin{enumerate}\renewcommand{\labelenumi}{$(\alph{enumi})$}
  \item the biholomorphism group $\mathsf{Aut}(\Omega_t)$ is trivial, and
\item there  is a biholomorphism $F_t :\Omega_t \rightarrow \hat{\Omega}_t$,
\end{enumerate}
then the map
$$
(p,t)  \mapsto F_t(p)
$$
is smooth.
\end{theorem}

The first two authors have previously established that, under the assumptions of Theorem~\ref{thm:smoothness of family of biholomorphisms}, the map
$$
(p,t)  \mapsto F_t(p)
$$
is continuous, see \cite[Proposition 1.3]{GG2020}.

For $t \in \Tc$, let $g_t$ denote the Bergman metric on $\Omega_t$ and let $\hat{g}_t$ denote the Bergman metric on $\hat{\Omega}_t$.  It is elementary that the Bergman metric is real analytic and it is a classical fact that on a strongly pseudoconvex domain the Bergman metric is complete.

Fix $t_0 \in \Tc$. By definition there are a neighborhood $\Tc_0$ of $t_0$, and smooth  maps $\Phi : \mathcal U\times \mathcal T_0 \rightarrow \mathbb C^d$ and $\hat{\Phi} : \hat{\mathcal U}\times\mathcal T_0 \rightarrow \mathbb C^d$ which satisfy Definition~\ref{defn: families of str pconvex domains} for $\{\Omega_t\}_{t \in \Tc}$ and $\{\hat{\Omega}_t\}_{t \in \Tc}$. Then let
$$
h_t : = \left(\Phi(\cdot ,t)|_{\Omega_{t_0}}\right)^* g_t \quad \text{and} \quad \hat{h}_t : = \left(\hat{\Phi}( \cdot,t )|_{\hat{\Omega}_{t_0}}\right)^* \hat{g}_t.
$$
Then $\{h_t\}_{t \in \Tc_0}$ and $\{ \hat{h}_t\}_{t \in \Tc_0}$ are families of complete real analytic Riemannian metrics on $\Omega_{t_0}$ and $\hat\Omega_{t_0}$ respectively (recall that being a real analytic metric means that there is some real analytic structure where the metric is real analytic). Further,
$$
G_t : = \left(\hat{\Phi}( \cdot,t )|_{\hat{\Omega}_{t_0}}\right)^{-1} \circ F_t \circ \left(\Phi( \cdot,t )|_{\Omega_{t_0}}\right)
$$
is an isometry between $(\Omega_{t_0}, h_t)$ and $(\hat \Omega_{t_0}, \hat h_t)$, and
$$
(p,t)  \mapsto G_t(p)
$$
is continuous.

Using Theorem~\ref{thm:continuity implies smoothness in riemannian case}, it then suffices to prove the following lemmas.

\begin{lemma} For all $t \in \Tc_0$, the groups $\mathsf{Isom}(\Omega_{t_0}, h_t) \cong \mathsf{Isom}(\Omega_t, g_t)$ and $\mathsf{Isom}(\hat{\Omega}_{t_0}, \hat h_t) \cong \mathsf{Isom}(\hat{\Omega}_t, \hat{g}_t)$ have order at most two. \end{lemma}

\begin{proof} By~\cite[Theorem 1.17]{GreeneKrantz} every isometry  for the Bergman metric on a strongly pseudoconvex domain  is either holomorphic or antiholomorphic. So the subgroups
$$
\mathsf{Aut}(\Omega_t) \leq \mathsf{Isom}(\Omega_t, g_t), \quad \mathsf{Aut}(\hat{\Omega}_t) \leq \mathsf{Isom}(\hat{\Omega}_t, \hat{g}_t)
$$
have index at most two. Since $\mathsf{Aut}(\Omega_t)$, $\mathsf{Aut}(\hat{\Omega}_t)$ are trivial, the lemma follows.
\end{proof}

\begin{lemma} The families of metrics $\{h_t\}$ and $\{\hat{h}_t\}$ are smooth. \end{lemma}

\begin{proof} It is a well known consequence of deep work of Hamilton, see Appendix~\ref{sec:smoothness of Bergman kernels}  for details, that the family of metrics $\{g_t\}$ and $\{\hat{g}_t\}$ are smooth. Hence $\{h_t\}$ and $\{\hat{h}_t\}$ are also smooth.
\end{proof}

\begin{lemma} There is a relatively compact connected open set $U \subset \Omega_{t_0}$ such that the inclusion map $U \hookrightarrow \Omega_{t_0}$ induces a surjection $\pi_1(U,u_0) \rightarrow \pi_1(\Omega_{t_0},u_0)$ for any $u_0 \in U$. \end{lemma}

\begin{proof}  Fix $\epsilon > 0$ sufficiently small and let
$$
U : = \{ z \in \Omega_{t_0} : {\rm dist}(z, \partial\Omega_{t_0}) > \epsilon\}.
$$
Using a normal neighborhood of $\partial\Omega_{t_0}$, we see that $\Omega_{t_0}$ deformation retracts onto $U$ and hence for any $u_0 \in U$ the inclusion map $U \hookrightarrow \Omega_{t_0}$ induces a surjection $\pi_1(U,u_0) \rightarrow \pi_1(\Omega_{t_0},u_0)$.
\end{proof}

\section{An example and questions}\label{sec: examples and conjecture}

In this section, we provide a construction which shows that the answer to Question~\ref{main question} is often no for non-rigid domains. Then we list some related open questions.

 \begin{proposition}\label{Omega 1/k}
Let $D$ be a strongly pseudoconvex domain in $\mathbb C^d$. If $\Aut(D) \neq \{\id_D\}$, then there are two smooth families  $\{\Omega_t\}_{t \in (-1,1)}$, $\{\hat \Omega_t\}_{t \in (-1,1)}$ of strongly pseudoconvex
  domains where $\Omega_0 = D = \hat \Omega_0$, $\Omega_t$ is biholomorphic to $\hat \Omega_t$ for all $t \in (-1,1)$,  and any family of biholomorphisms $\{ F_t : \Omega_t \rightarrow \hat \Omega_t\}_{t \in (-1,1)}$ is discontinuous at $t=0$.
\end{proposition}

We sketch the proof of Proposition~\ref{Omega 1/k} and then provide a rigorous argument at the end of the section.

\begin{proof}[Proof Sketch] Burns--Shnider--Wells proved, in a very precise sense, that rigid domains are dense in the space of all domains. Using their result, we will construct a smooth family of domains $\{D_t\}_{t \in [0,1]}$ where
\begin{enumerate}\renewcommand{\labelenumi}{(\roman{enumi})}
\item $D_0  = D$,
\item $D_{1/k}$ is rigid when $k =1,2,\dots$, and
\item $D_t \subset D$ for all $t \in [0,1]$.
\end{enumerate}
We then fix $f \in \Aut(D) \setminus \{\id_D\}$ and define two families of domains
$$
\Omega_t : = \begin{cases} D_{e^{1-1/\abs{t}}} & \text{ if } t \neq 0 \\ D & \text{ if } t = 0 \end{cases}
$$
and
$$
\hat{\Omega}_t : = \begin{cases} D_{e^{1-1/\abs{t}}} & \text{ if } t > 0 \\ D & \text{ if } t = 0 \\ f(D_{e^{1-1/\abs{t}}}) & \text{ if } t < 0\end{cases}.
$$
Since $t \mapsto e^{-1/\abs{t}}$ vanishes to infinite order at $t=0$, both of the above families are smooth. Further, for each $t \in (-1,1)$ the domains $\Omega_t$ and $\hat \Omega_t$ are biholomorphic. However, since $D_{1/k}$ is rigid, if $\{F_t : \Omega_t \rightarrow \hat\Omega_t\}$ is a family of biholomorphisms, then $F_{1/(1+\log(k))} = \id$ and $F_{-1/(1+\log(k))} = f$ when $k=1,2,\dots$. So any such family is discontinuous at $t=0$.
\end{proof}

In the special case where $\Omega_0$ is the unit ball, Proposition~\ref{Omega 1/k} was proved in~\cite{GG2020}.
In all of these  examples,
 the automorphism group of the domains does not vary continuously in $t$, which leads to the following natural question.

\begin{question}
Suppose $\{\Omega_t\}_{t\in\mathcal T}$ and $\{\hat\Omega_t\}_{t\in\mathcal T}$ are two smooth families of strongly pseudoconvex domains in $\mathbb C^d$ where for each $t\in\mathcal T$ the domains $\Omega_t$ and $\hat \Omega_t$ are biholomorphic. If the groups $\mathsf{Aut}(\Omega_t)$ are all isomorphic, is there a smooth family  $\{ F_t : \Omega_t \rightarrow \hat \Omega_t\}_{t \in \Tc}$ of biholomorphism?
\end{question}

Theorem~\ref{thm:smoothness of family of biholomorphisms} provides a positive answer to the above question when the automorphism groups are trivial. In the case when each $\mathsf{Aut}(\Omega_t)$  is non-compact, the Wong and Rosay ball theorem~\cite{Wong-77,Rosay-79} implies that $\Omega_t$ is biholomorphic to the unit ball and in this case the above question also has a positive local answer.

\begin{proposition}\label{prop:ball example} Suppose $\{\Omega_t\}_{t\in\mathcal T}$ and $\{\hat\Omega_t\}_{t\in\mathcal T}$ are two smooth families of strongly pseudoconvex domains in $\mathbb C^d$ where for each $t\in\mathcal T$ the domains $\Omega_t$ and $\hat \Omega_t$ are biholomorphic. If each $\Aut(\Omega_t)$ is non-compact (i.e. $\Omega_t$ is biholomorphic to the unit ball), then each $t_0 \in \Tc$ has an open neighborhood $\Tc_0 \subset \Tc$ where there exists a smooth family  $\{ F_t : \Omega_t \rightarrow \hat \Omega_t\}_{t \in \Tc_0}$ of biholomorphisms.
\end{proposition}

\begin{proof} See Subsection~\ref{sec:proof of ball example} below.
\end{proof}

\begin{remark}
When $\{\Omega_t\}_{t\in \Tc}$ is a holomorphic family of {\it compact} complex manifolds where each $\Omega_t$ is biholomorphic to $\Omega_{t_0}$,  Fischer and Grauert~\cite{local-trivial} showed that locally in $t$ there is a holomorphic family of biholomorphisms sending $\Omega_t$ onto $\Omega_{t_0}$.
\end{remark}

   It would be interesting to  identify
    topological conditions on $\Tc$ which imply the existence of a global smooth family of biholomorphisms in the setting of Proposition~\ref{prop:ball example}.

Question~\ref{main question} is stated for smooth families of   complex manifolds,
 but one could also consider the analogous question with other regularity assumptions and conclusions. For instance,  is the real analytic version of Theorem~\ref{thm:smoothness of family of biholomorphisms} true?

\begin{question}
Suppose $\{\Omega_t\}_{t\in\mathcal T}$ and $\{\hat\Omega_t\}_{t\in\mathcal T}$ are two real analytic families of real analytic strongly pseudoconvex domains in $\mathbb C^d$ where for each $t\in\mathcal T$ the domains $\Omega_t$ and $\hat \Omega_t$ are biholomorphic. If the groups $\mathsf{Aut}(\Omega_t)$ are all trivial, is the unique family  $\{ F_t : \Omega_t \rightarrow \hat \Omega_t\}_{t \in \Tc}$ of biholomorphisms real analytic in $t$?
\end{question}

In the context of the above question, it is worth noting that the methods used in the proof of Proposition~\ref{Omega 1/k} cannot be used to produce a real analytic family of domains.

It also seems natural to wonder if a positive answer to Question~\ref{main question} in the continuous category implies a positive answer in the smooth category.

\begin{question}
Suppose $\{\Omega_t\}_{t\in\mathcal T}$ and $\{\hat\Omega_t\}_{t\in\mathcal T}$ are two smooth families of strongly pseudoconvex domains in $\mathbb C^d$ where for each $t\in\mathcal T$ the domains $\Omega_t$ and $\hat \Omega_t$ are biholomorphic. If there is a continuous family of
 biholomorphisms  $\{ F_t : \Omega_t \rightarrow \hat \Omega_t\}_{t \in \Tc}$, then is there a smooth family of biholomorphisms?
\end{question}

\subsection{Proof of Proposition~\ref{Omega 1/k}} Suppose $D$ is a strongly pseudoconvex domain in $\mathbb C^d$ with $\Aut(D) \neq \{\id_D\}$. Fix a strongly plurisubharmonic defining function $\rho_0 : V \rightarrow \Rb$ of $\partial D$, where $V$ is a neighborhood of $\partial D$, $\nabla \rho_0 \neq 0$ on $\partial D$,   $D \cap V =\{\rho_0 < 0\}$  and $\partial D = \{ \rho_0 = 0\}$. Then fix $\delta > 0$ sufficiently small so that if
$$
U : = \{ -\delta < \rho_0 < \delta\},
$$
then $\nabla \rho_0 \neq 0$ on $\overline{U}$ and $\overline{U}$ is a smooth manifold with boundary inside of $V$.

Let $\Pc(\overline{U})$ denote the open set in $C^\infty(\overline{U})$ of strongly plurisubharmonic functions $\psi : \overline{U} \rightarrow \Rb$ where $\nabla \psi \neq 0$ on $\overline{U}$. Fix $\epsilon_0 > 0$ sufficiently small so that: if $\psi \in \Pc(\overline{U})$ and $\abs{\rho_0 - \psi}_{U,2} < \epsilon_0$, then  the set
$$
D_\psi : = \{ z \in U : \psi(z) < 0\} \cup (D \setminus U)
$$
is a strongly pseudoconvex domain. Then let
$$
\Nc : = \left\{ \psi \in \Pc(\overline{U}) : \abs{\rho_0 - \psi}_{U,2} < \epsilon_0 \right\}.
$$

We will use the following density theorem of Burns--Shnider--Wells.

\begin{theorem}[{\cite[Theorem 4.1]{BSW}}] For any $\epsilon > 0$, $k \in \mathbb{N}$, and $\psi \in \Nc$ there exists $\tilde{\psi} \in \Nc$ such that
$$
\abs{\tilde \psi - \psi}_{U,k} < \epsilon
$$
and $\mathsf{Aut}(D_{\tilde{\psi}})$ is trivial.
\end{theorem}

Fix $\epsilon > 0$ sufficiently small so that $\rho_0 + \epsilon t \in \Nc$ for all $t \in [0,1]$. Then for each $k \in \mathbb{N} \setminus \{0\}$, we can apply the Burns--Shnider--Wells theorem to $\rho_0 + \frac{\epsilon}{k}$,  to find $\tilde\rho_{k} \in \Nc$ such that   $\mathsf{Aut}(D_{\tilde\rho_k})$   is trivial and
\begin{equation}
\label{eqn:estimate on rho_1/k}
\abs{\tilde\rho_{k}-\rho_0 - \frac{\epsilon}{k}}_{U,\max\{k,2\}}<\frac{\epsilon}{k+2^k}.
\end{equation}

Fix a smooth function  $\chi\colon \Rb\to [0,1]$ with $\supp\chi\subset(-1,1)$ and $\chi(0)=1$. For each $k \in \Nb\setminus\{0\}$, define   $\chi_k : \Rb \rightarrow [0,1]$ by
$$
\chi_k(t) =\chi\left(2k(k+1)\left(t-\frac{1}{k}\right)\right).
$$
Then
$$
\chi_k\left(\frac{1}{k}\right)=1\quad \text{and} \quad \chi_k(t)=0 \quad \text{when} \quad \abs{t-\frac{1}{k}}>\frac{1}{2k(k+1)}.
$$
Further, there exists an increasing sequence of constants $\{C_j\}_{j \in \mathbb{N}}$ such that
\begin{equation}\label{eqn:estimate on chi derivatives}
\abs{ \chi_k^{(j)}
(t)
}
 \leq C_j k^{2j}
\end{equation}
for all $t\in\Rb$,  all $k \in \mathbb{N} \setminus \{0\}$, and all $j \in \Nb$.

Define $\rho : U \times [0,1] \rightarrow \Rb$ by
$$
\rho(\cdot,t) = \rho_0+\epsilon t - \sum_{k =1}^\infty \chi_k(t) \left( \tilde\rho_k - \rho_0-\frac{\epsilon}{k} \right).
$$
Notice that when $k \geq m$, Estimates ~\eqref{eqn:estimate on rho_1/k} and~\eqref{eqn:estimate on chi derivatives} imply that
\begin{align*}
\abs{\chi_k(t) \left( \tilde\rho_k - \rho_0-\frac{\epsilon}{k} \right)}_{U \times [0,1], m}& \leq \sum_{j=0}^m \begin{pmatrix} m \\ j \end{pmatrix} \abs{\chi_k}_{[0,1], j} \abs{  \tilde\rho_k - \rho_0-\frac{\epsilon}{k} }_{U, m-j} \\
& \leq 2^m C_m \frac{k^{2m}}{k+2^k}\epsilon.
\end{align*}
Then, since the $\{ \chi_k\}$ have disjoint support,
$$
\abs{\rho}_{U \times [0,1], m} \leq \abs{ \rho_0+\epsilon t - \sum_{k =1}^{m-1} \chi_k(t) \left( \tilde\rho_k - \rho_0-\frac{\epsilon}{k} \right)}_{U \times [0,1], m} +2^m C_m\epsilon\max_{k \geq m}  \frac{ k^{2m}}{k+2^k}
$$
for $m \in \Nb$. Hence $\rho \in C^\infty(U \times [0,1])$.

Also, notice that
\begin{equation}\label{eqn:C2 norm is bounded by epsilon}
\abs{\rho(\cdot,t) - \rho_0}_{U,2} \leq\epsilon \left(1 +  4C_2
\max_{k \geq 1}  \frac{  k^4
}{k+2^k}\right)
\end{equation}
for all $t \in [0,1]$.  Hence, by shrinking $\epsilon > 0$ we may further assume that $\rho(\cdot,t) \in \Nc$ for every $t \in [0,1]$.

Let $D_t : = D_{\rho(\cdot,t)}$. Then $D_0 = D$ and $D_{\frac{1}{k}} = D_{\tilde\rho_{k}}$ for $k \in \Nb \setminus \{0\}$.

\begin{lemma} $D_t \subset D$ for all $t \in [0,1]$.
\end{lemma}

\begin{proof} It suffices to fix $z \in \partial D$ and show that $\rho(t,z) \geq 0$. If $t \in {\rm supp}(\chi_k)$, then
\begin{align*}
\rho(t,z) & = \rho_0(z)+\epsilon t - \chi_k(t) \left(  \tilde\rho_{k}(z)-\rho_0(z)-\frac{\epsilon}{k} \right) \geq 0+\epsilon t - \frac{\epsilon}{k+2^k}  \\
& \geq
 \frac{\epsilon}{k}
 - \frac{\epsilon}{2k(k+1)} - \frac{\epsilon}{k+2^k} > 0.
\end{align*}
Otherwise, if $t \notin \cup_{k=1}^\infty {\rm supp}(\chi_k)$, then
$$
\rho(t,z) = \rho_0(z) + \epsilon t = 0 + \epsilon t \geq 0.
$$
Hence $D_t \subset D$ for all $t \in [0,1]$. \end{proof}

Now fix $f \in \Aut(D)$ non-trivial. Define families of domains
$$
\Omega_t : = \begin{cases} D_{e^{1-1/\abs{t}}} & \text{ if } t \in [-1,1] \setminus \{0\} \\ D & \text{ if } t = 0 \end{cases}
$$
and
$$
\hat{\Omega}_t : = \begin{cases} D_{e^{1-1/\abs{t}}} & \text{ if } 0 < t \leq 1 \\ D & \text{ if } t = 0 \\ f(D_{e^{1-1/\abs{t}}}) & \text{ if } -1 \leq t < 0\end{cases}.
$$

\begin{lemma} After possibly shrinking $\epsilon > 0$, the families $\{\Omega_t\}_{t\in(-1,1)}$ and $\{\hat \Omega_t\}_{t\in(-1,1)}$ are smooth.\end{lemma}

\begin{proof}
     Choose  a smooth bump function $\chi\colon\Cb^d\to[0,1]$ with
  $\supp \chi\subset D_0$ and $\chi=1$ on  $D_0\setminus U$.
     Set $\tilde U=D_0\cup U$. On $\tilde U\times[0,1]$, define
\begin{equation}\label{def-ext}
\rho^*(z,t)=-\chi+(1-\chi) \rho(z,t).
\end{equation}

By Equation~\eqref{eqn:C2 norm is bounded by epsilon}, we may shrink $\epsilon > 0$ so that the function $\rho^*$ satisfies the hypothesis of Lemma~\ref{rho-t}  and hence $\{D_t\}_{t\in [0,1]}$ is a smooth family of domains. Then there are a neighborhood $\mathcal U_0$ of $\overline{D}$ and a   smooth map $\Phi^* : \mathcal U_0 \times [0,1] \rightarrow \mathbb C^d$ such that for every $t \in  [0,1]$ the map $\Phi^*(\cdot,t ) : \mathcal U \rightarrow \Cb^d$ is a diffeomorphism onto its image and $\Phi^*(\cdot,t)(\overline{D_0}) = \overline{D_t}$.

Define $\Phi: \mathcal U_0 \times (-1,1) \rightarrow \mathbb C^d$ by
$$
\Phi(\cdot, t)  = \begin{cases} \Phi^*(\cdot, e^{1-1/\abs{t}}) & \text{ if } t \in (-1,1) \setminus \{0\} \\ \Phi^*(\cdot, 0) & \text{ if } t = 0 \end{cases}.
$$
Then $\Phi$ is smooth and $\Phi(\cdot,t)(\overline{\Omega_0})=\overline{\Omega_t}$ for all $t \in (-1,1)$. Hence $\{\Omega_t\}_{t\in(-1,1)}$ is a smooth family.

By Fefferman's theorem~\cite{fefferman} $f$ extends smoothly to $\overline{D}$. Then, after possibly shrinking $\mathcal U_0$, we can assume that $f$ extends to a smooth map $\tilde f : \mathcal U_0 \rightarrow \Cb^d$, which is a diffeomorphism onto its image.

Next, we can find a smaller neighborhood $\mathcal U_1 \subset \mathcal U_0$ of $\overline{D}$ such that  the map
$$
(z, t) \mapsto\Big( \tilde f\circ \Phi(\cdot,t)\circ(\Phi(\cdot,0))^{-1}\circ \tilde f^{-1}\circ\Phi(\cdot,0)\Big)(z,t)
$$
is well defined and smooth on $\mathcal U_1 \times (-1,0]$. Then define $\hat\Phi: \mathcal U_1 \times (-1,1) \rightarrow \mathbb C^d$ by
$$
\hat\Phi(\cdot,t)=  \begin{cases} \Phi(\cdot, t) & \text{ if } 0 \leq t < 1 \\ \tilde f\circ \Phi(\cdot,-t)\circ(\Phi(\cdot,0))^{-1}\circ \tilde f^{-1}\circ\Phi(\cdot,0) & \text{ if } -1 < t \leq 0 \end{cases}.
$$
Notice that $\hat \Phi$ is smooth since $t \mapsto e^{-1/\abs{t}}$ vanishes to infinite order at $t = 0$. Further, $\hat\Phi(\cdot,t)(\overline{\hat\Omega_0})=\overline{\hat\Omega_t}$. Hence $\{\hat\Omega_t\}_{t\in(-1,1)}$ is a smooth family. \end{proof}

Now suppose that $\{F_t : \Omega_t \rightarrow \hat{\Omega}_t\}_{t \in (-1, 1)}$ is a family of biholomorphisms.  By construction $\Omega_{\pm 1/(1+\log(k))}=D_{1/k}$ has trivial automorphism group for all $k =1,2, \dots$. Hence we must have
$$
 F_{1/(1+\log(k))} = \id \quad \text{and} \quad F_{-1/(1+\log(k))} = f.
$$
So $\{F_t : \Omega_t \rightarrow \hat{\Omega}_t\}_{t \in (-1, 1)}$ is discontinuous at $t=0$.

\subsection{Proof of Proposition~\ref{prop:ball example}}\label{sec:proof of ball example} Let $g_t$ and $\hat{g}_t$ be the Bergman metrics on $\Omega_t$ and $\hat\Omega_t$, respectively. These metrics are smooth in $t$ by work of Hamilton, see Corollary~\ref{Bergman-kernel}. Then we can find a neighborhood $\Tc_0$ of $t_0$ in $\Tc$, points $p_0,q_0 \in \Cb^d$, smooth maps $v_1,\dots, v_d : \Tc_0 \rightarrow T_{p_0} \Cb^d$, and smooth maps $w_1,\dots, w_d : \Tc_0 \rightarrow T_{q_0} \Cb^d$ such that: if $t \in \Tc_0$, then
\begin{enumerate}\renewcommand{\labelenumi}{(\roman{enumi})}
\item $p_0\in \Omega_t$ and $v_1(t), \dots, v_d(t)$ is an unitary basis of $T_{p_0} \Cb^d$ with respect to $g_t$,
\item $q_0\in \hat{\Omega}_t$ and $w_1(t), \dots, w_d(t)$ is an unitary basis of $T_{q_0} \Cb^d$ with respect to $\hat{g}_t$.
\end{enumerate}
Since $\Omega_t$ and $\hat \Omega_t$ are both biholomorphic to the unit ball, there exists a holomorphic isometry $F_t : (\Omega_t, g_t) \rightarrow (\hat\Omega_t, \hat g_t)$ such that $F_t(p_0) = q_0$ and
$$
d(F_t)_{p_0} v_j(t) = w_j(t)
$$
for $j=1,\dots, d$. Clearly, $d(F_t)_{p_0}$  is smooth in $t$.
Since $\injrad_{{g}_{t}}(p_0)=\infty$,  Formula \eqref{efp} and Observation~\ref{obs:smoothness of exponential maps}  imply that
$$
F_t=\exp^{\hat{g}_t}_{q_0} \circ d(F_t)_{p_{0}} \circ (\exp^{g_{t}}_{p_{0}})^{-1}
$$
 is a smooth family of biholomorphisms.

\appendix

\section{Smooth families of  domains}\label{sec:appendix}

In this Appendix we prove a technical result about smooth families of domains. In particular, we will show that one could equivalently define a smooth family of strongly  pseudoconvex  domains in terms of smooth families of defining functions. This was used in Section~\ref{sec: examples and conjecture}.

\begin{lemma}\label{rho-t} Suppose $U \subset \Cb^d$ is an open relatively compact subset. Assume $\rho: U\times\mathcal T \rightarrow \Rb$ is a smooth function where for each $t \in \Tc$:
\begin{enumerate}\renewcommand{\labelenumi}{$(\alph{enumi})$}
\item $\Omega_{t}:=\{z\in U\colon \rho(   z
,t)<0\}$ is a strongly pseudoconvex domain,
\item $\overline{\Omega_t} \subset U$,
\item $\nabla_z\rho(\cdot, t) \neq0$ on the zero set of $\rho(\cdot, t)$, and
\item $\rho(\cdot,t)$ is strongly plurisubharmonic in a neighborhood of $\partial \Omega_t$.
\end{enumerate}
Then $\{\Omega_t\}_{t \in \Tc}$ is a smooth family of strongly pseudoconvex domains.
\end{lemma}

\begin{proof} Smoothness is a local property and so we may assume that $\Tc$ is an open connected set in $\Rb^k$ for some $k \geq 1$. Let
$$
M : = \{ (z,t) \in U \times \Tc : \rho(z,t) = 0\}.
$$
Since $\nabla_{z,t} \rho \neq 0$ in a neighborhood of $M$, the implicit function theorem implies that $M$ is a smooth submanifold of  $\Cb^d \times \Rb^k$ with
$$
M \cap (\Cb^d \times \{t\}) = \partial \Omega_t \times \{t\}
$$
for all $t \in \Tc$. Further, since $\nabla_z\rho \neq0$ on $M$, the projection $\pi : M \rightarrow \Tc$ is a proper submersion. So by Ehresmann's fibration theorem, $\pi$ is a locally trivial fibration.

Now fix $t_0 \in \Tc$. Since $\pi$ is a locally trivial fibration,  there is an open neighborhood $\Tc_0$ of $t_0$ in $\Tc$ and there is a diffeomorphism $F : \partial\Omega_{t_0} \times \Tc_0 \rightarrow \pi^{-1}(\Tc_0)$ with $F(\cdot, t_0) = \id_{\partial\Omega_{t_0}}$ and
$$
F\left(\partial \Omega_{t_0} \times\{ t\}\right) = \partial \Omega_t \times \{t\}
$$
for each $t \in \Tc$.

For each $t \in \Tc$ and $z_0 \in \partial\Omega_t$, let
$$
\textbf{n}_t(z_0) = \frac{ - \nabla_z \rho(z_0, t)}{\norm{\nabla_z \rho(z_0, t)}}
$$
denote the inward pointing unit normal vector of $\partial\Omega_t$ at $z_0$. Since $\partial\Omega_{t_0}$ is a smooth hypersurface, we can find a neighborhood $\Vc$ of $\partial\Omega_{t_0}$ such that the  closest point projection $\pi_{t_0} : \Vc \rightarrow \partial\Omega_{t_0}$ is smooth. Then let $\delta_{t_0} : \Vc \rightarrow \Rb$ be the signed distance to $\partial\Omega_{t_0}$, that is the unique map where
$$
x = \pi_{t_0}(x) +\delta_{t_0}(x) \textbf{n}_{t_0}(\pi_{t_0}(x))
$$
for all $x \in \Vc$.

Fix a smooth bump function $\chi : \Cb^d \rightarrow [0,1]$ such that ${\rm supp}(\chi) \subset\Vc$ and $\chi \equiv 1$ on a relatively compact open neighborhood $\Vc_0 \subset \Vc$ of $\partial\Omega_{t_0}$. By shrinking $\Tc_0$, we may assume that $\partial\Omega_t \subset \Vc_0$ for all $t \in \Tc_0$. Then define $\Phi : \Cb^d \times \Tc_0 \rightarrow \Cb^d$ by
$$
\Phi(z,t) = (1-\chi(z))z + \chi(z)\Big( F(\pi_{t_0}(z), t) + \delta_{t_0}(z) \textbf{n}_t\big( F(\pi_{t_0}(z),  t)\big) \Big).
$$
Then $\Phi$ is smooth and $\Phi(\cdot, t_0) = \id_{\Cb^d}$. So by shrinking $\Tc_0$ further, we may assume that $\Phi(\cdot, t)|_{\Vc_0}$ is a diffeomorphism onto its image for all $t \in \Tc_0$. Further, since $\Phi(\cdot, t)|_{\partial \Omega_{t_0}} = F(\cdot, t)$, we have $\Phi(\cdot, t)(\overline{\Omega_{t_0}}) = \overline{\Omega_t}$. Thus $\{\Omega_t\}_{t \in \Tc}$ is a smooth family.
\end{proof}

\section{Smoothness of Bergman kernels in families of domains}\label{sec:smoothness of Bergman kernels}

In this Appendix we explain how to use a result of Hamilton to deduce that the Bergman kernel varies smoothly in a smooth family of strongly pseudoconvex domains. This consequence of Hamilton's work is well known, see for instance the discussion in~\cite{Komatsu} or ~\cite{wang2014variation}, but for the reader's convenience we provide a detailed exposition of this application.

\subsection{Hamilton's abstract result}\label{appendix:Hamilton's abstract result}

In this section we recall a result of Hamilton. In the discussion that follows, given a vector bundle $E \rightarrow X$, we let $\mathscr{C}(X;E)$ denote the  Fr\'echet space of smooth sections.

Now fix  $r \in \Nb$, a compact manifold $X$ with non-empty boundary, vector bundles $F,M$ over $X$, and vector bundles $P,Q$ over $\partial X$ where
\begin{enumerate}\setcounter{enumi}{-1}
\item ${\rm rank} \, P + {\rm rank} \, Q = {\rm rank} \, F$.
\end{enumerate}
We consider a family of operators
$$
(\mathcal{E}(m))_{m \in \Oc}=  (E(m), p(m), q(m))_{m \in \Oc}:  \mathscr{C}(X;F) \rightarrow  \mathscr{C}(X;F)\times  \mathscr{C}(\partial X;P) \times  \mathscr{C}(\partial X;Q)
$$
indexed by an open neighborhood $\Oc \subset \mathscr{C}(X;M)$ of the zero section satisfying the eight properties listed below.

We suppose that
\begin{enumerate}
\item $E(m)$ is a linear partial differential operator of degree 2,
\item $p(m)$ is a linear partial differential operator of degree 0,
\item $q(m)$ is a linear partial differential operator of degree  1, and
\item the coefficients of these operators in any local trivialization depend smoothly on $m$ and its derivatives up to degree $r$.
\end{enumerate}
We also suppose that each of these vector bundles is endowed with a family of Hermitian metrics, all denoted by $(\ip{\cdot, \cdot}_m)_{m \in \Oc}$, whose coefficients in any local trivialization depend smoothly on $m$
   and its derivatives up to degree $r$. We further suppose that with respect to these inner products the principal
symbols $\sigma_{E(m)}$, $\sigma_{p(m)}$, and $\sigma_{q(m)}$ of $E(m)$, $p(m)$, and $q(m)$ satisfy the following:
\begin{itemize}
\item When $m \in \Oc$, $z \in X$, $f,g \in F_z$, and $\xi \in T^*_z X$
\begin{enumerate}\setcounter{enumi}{4}
\item $\ip{ \sigma_{E(m)}(\xi) f, f}_m > 0$ with equality if and only if $\xi = 0$ or $f = 0$,
\item $\ip{ \sigma_{E(m)}(\xi) f, g}_m = \ip{ f, \sigma_{E(m)}(\xi) g}_m$.
\end{enumerate}
\item There exists a normal positive covector field $ \nu_m
 \in \mathscr{C}(\partial X, T^* X|_{\partial X})$   where $\nu_m$ depends on $m$ and its derivatives up to degree $r$ such that: when $m \in \Oc$, $z \in \partial X$, $f,g \in F_z$, and $\xi \in T^*_z X$
\begin{enumerate}\setcounter{enumi}{6}
\item $\ip{ \sigma_{E(m)}(\nu_m) f, g}_m = \ip{ \sigma_{p(m)} f, \sigma_{p(m)} g}_m+\ip{ \sigma_{q(m)}(\nu_m) f, \sigma_{q(m)}(\nu_m) g}_m$,
\item if $\sigma_{p(m)} f = \sigma_{q(m)} g = 0$, then
$$
\ip{ D\sigma_{E(m)}(\nu_m; \eta) f, g}_m = \ip{ \sigma_{q(m)}(\eta) f, \sigma_{q(m)}(\nu_m) g}_m+\ip{ \sigma_{q(m)}(\nu_m) f, \sigma_{q(m)}(\eta) g}_m.
$$
\end{enumerate}
\end{itemize}
Recall that when $z \in \partial X$, a cotangent vector $f \in T_z^* X$ is \emph{normal} if $f|_{T_z \partial X} \equiv 0$ and is \emph{positive} if $f(v)  > 0$ when $v \in T_z X$ is inward pointing.

\begin{remark} When $X = \overline{\Omega}$ where $\Omega$ is a smoothly bounded domain in a larger manifold (e.g. $\Cb^d$) and $\rho$ is a defining function for $\Omega$, then  $\nu_m=-a_md\rho$ for some positive function $a_m : \partial X \rightarrow (0,\infty)$. Further, we may assume that $a_m=1$ by scaling the metrics on $P$ without changing the metrics on $F,Q$.
\end{remark}

We also suppose that $X$ has a family of volume forms  $(dV_m)_{m \in \Oc}$
 which also depends smoothly on $m$  and its derivatives up to degree $r$. Then we can define a family of inner products on $\mathscr{C}(X;F)$ by
$$
\langle\langle f,g \rangle\rangle_m = \int_X \ip{f,g}_m dV_m.
$$
Let $\norm{\cdot}_{0}$ denote the norm on $\mathscr{C}(X;F)$ associated  to $\langle\langle \cdot, \cdot \rangle\rangle_0$.

Let  $dS_0$
 be a volume form on $\partial X$ and then given $f \in\mathscr{C}(X;F|_{\partial X})$, define
$$
\abs{f}_0 : = \sqrt{  \int_{\partial X} \ip{f,f}_0 dS_0}
$$
(notice that changing the volume form $dS_0$ produces an equivalent norm).

The family $\mathcal{E}$ satisfies an \emph{uniform persuasive estimate} if there exists $C > 0$ such that
$$
\abs{f}_0^2 \leq C\left( {\rm Re} \langle\langle E(m)f,f \rangle\rangle_m + \norm{f}_{0}^2\right)
$$
for all $m \in \Oc$ and $f \in \mathscr{C}(X;F)$ satisfying $p(m)f = 0$ and $q(m)f = 0$.

Let
$$
H(m) : = \{ f \in \mathscr{C}(X;F)  : E(m)f = 0, \, p(m)f =0, \, q(m) f = 0\}.
$$

\begin{theorem}[\cite{Hamilton1}, p.~438] \label{Hamilton-result}
Suppose that $\mathcal  E$ satisfies a uniform persuasive estimate, and $H(0)=0$. Then
 $H(m)=0$ for all $m$ in a possibly smaller neighborhood $\mathcal O_0 \subset\mathcal O$ of the zero section in $\mathscr{C}(X;M)$ and hence for each $m \in \mathcal O_0$,
  the linear operator $\mathcal  E_m$ is invertible; moreover if we define
$$
\mathcal  E^{-1}\colon (
\mathcal O_0\subset \mathscr{C}(X;M)
)\times (\mathscr{C}(X;F)\oplus \mathscr{C}(\partial X;P)\oplus \mathscr{C}(\partial X;Q))\to \mathscr{C}(X;F)
$$
by letting $\mathcal  E^{-1}(m)(g,h,k)=f$ be the solutions of $\mathcal  E(m)f=(g,h,k)$, then the family of inverse $\mathcal  E^{-1}$ is a smooth tame map.
\end{theorem}

\subsection{The $\bar{\partial}$-Neumann operator on a strongly pseudoconvex domain}\label{appendix:dbar Neumann}

 In this section we recall the definition of the $\bar{\partial}$-Neumann operator and then place this operator into Hamilton's framework. For more details on the $L^2$ theory of the $\bar{\partial}$-Neumann operator, see ~\cite{Chen-Shaw}.

For the rest of the section fix a strongly pseudoconvex domain $\Omega \subset \Cb^d$ and let $\rho : \Cb^d \rightarrow \Rb$ be a defining function for $\Omega$ with $\norm{\nabla \rho} = 1$ on $\partial \Omega$.
  Thus    $-d\rho$  is the unique normal covector satisfying $d\rho(\nabla\rho)=1$ on $\partial\Omega$.

Given a smooth $(p,q)$-form $f=\sum^\prime_{\abs{I}=p,\abs{J}=q} f_{I,J} dz^I \wedge d\bar{z}^J$, define
$$
\bar{\partial} f = \sideset{}{'}\sum_{\abs{I}=p,\abs{J}=q} \sum_j \frac{\partial f_{I,J}}{\partial \bar{z}^j} d\bar{z}^j \wedge dz^I \wedge d\bar{z}^J.
$$
One can view $\bar{\partial}=\bar{\partial}_{p,q} : L^2_{(p,q)}(\Omega) \rightarrow L^2_{(p,q+1)}(\Omega)$ as a densely defined operator from the space of $L^2$-integrable $(p,q)$-forms on $\Omega$ to the space of $L^2$-integrable $(p,q+1)$-forms on $\Omega$. Then let $\bar{\partial}^* :  L^2_{(p,q+1)}(\Omega) \rightarrow  L^2_{(p,q)}(\Omega)$ denote the $L^2$-adjoint. The associated Laplacian
$$
\square= \bar{\partial}\bar{\partial}^*+\bar{\partial}^*\bar{\partial}: L^2_{(p,q)}(\Omega) \rightarrow L^2_{(p,q)}(\Omega)
$$
has domain
$$
\left\{ f \in L^2_{(p,q)}(\Omega) : f \in {\rm dom}(\bar{\partial}) \cap {\rm dom}(\bar{\partial}^*), \, \bar\partial f \in  {\rm dom}(\bar{\partial}^*),  \bar\partial^* f \in  {\rm dom} (\bar{\partial})\right\}.
$$
This Laplacian has a bounded inverse $N_{p,q}  : L^2_{(p,q)}(\Omega)  \rightarrow L^2_{(p,q)}(\Omega)$ which is called the \emph{$\bar\partial$-Neumann operator}, see ~\cite[Section 4.4]{Chen-Shaw}.

Remarkably, the Bergman kernel can be recovered from this operator. Fix a smooth function $\chi : \Rb \rightarrow [0,1]$ such that $\chi \equiv 1$ near 0, ${\rm supp}(\chi) \subset [-1,1]$, and
$$
\int_{\Cb^d} \chi(\abs{z}) dV = 1.
$$
Then fix $w \in \Omega$ and $\epsilon > 0$ such that the Euclidean ball $B_{\Cb^d}(w,\epsilon)$ is contained in $\Omega$. Then
\begin{equation}\label{eqn:Bergman kernel in terms of dbar Neumann}
K_\Omega(\cdot, w) = f_w - \bar\partial^* N_{0,1}\bar\partial f_w,
\end{equation}
where $f_w(z) = \frac{1}{\epsilon^{2d}}\chi(\abs{z-w}/\epsilon)$.  For a proof of the above formula of Kerzman and Bell~\cite{kerzman, bell},
see ~\cite[Theorem 4.4.5 and p.~147]{Chen-Shaw}.

Next we place the $\bar\partial$-Neumann operator into Hamilton's framework. To match the notation in the previous section, we let $X : = \overline{\Omega}$.   Let $ \Lambda^{1}(X) =(T^* X)\otimes\mathbb C$.
Then consider the vector bundles
\begin{align*}
F & := \Lambda^{(0,1)}(X) = (  \Lambda^{1}(X))^{(0,1)}, \quad P  := \partial X \times \Cb, \text{ and} \\
Q & := \left\{ \sum c_j d\bar{z}^j  \in \Lambda^{(0,1)}(X)|_{\partial X} : \sum c_j  \frac{\partial \rho}{\partial z^j} = 0 \right\}
\end{align*}
over $X$, $\partial X$, and $\partial X$ respectively. Notice that ${\rm rank} \, P + {\rm rank} \, Q = {\rm rank} \, F$.

Given a $(p,q)$-form $f=\sum^\prime_{\abs{I}=p,\abs{J}=q} f_{I,J} dz^I \wedge d\bar{z}^J$, define
$$
\vartheta f=-\sum_j \sideset{}{^\prime}\sum_{|K|=q-1}  \frac{\partial f_{jK}}{\partial z^j}d\bar z^K.
$$
Then define operators $E :  \mathscr{C}(X;F) \rightarrow \mathscr{C}(X;F)$,  $p : \mathscr{C}(X;F) \rightarrow \mathscr{C}(X;P)$, and $q : \mathscr{C}(X;F) \rightarrow \mathscr{C}(X;Q)$ by
\begin{align*}
\tilde{E}(f) & = \left( \bar{\partial}\vartheta +\vartheta\bar{\partial}\right)f = \sum_j -\frac{1}{4}\Delta f_j d\bar z^j, \quad
\tilde{p}(f)  = \sum_{j} f_j \frac{\partial \rho}{\partial z^j}, \text{ and} \\
\tilde{q}(f) & =  2\sum_{k,j} \left( \frac{\partial f_k}{\partial \bar{z}^j} - \frac{\partial f_j}{\partial \bar{z}^k} \right)  \frac{\partial \rho}{\partial z^j} d\bar{z}^k
\end{align*}
when $f = \sum_j f_j d\bar z^j$.

The vector bundle $\Lambda^1(\Cb^d) \rightarrow \Cb^d$ has a standard Hermitian inner product where $dz^1, \dots, dz^d, d\bar{z}^1, \dots, d\bar{z}^d$ is a unitary basis of each fiber.  This induces Hermitian inner products  on the vector bundles $F$ and $Q$.
 We can also endow vector bundle $P = \partial X \times \Cb$ with a standard inner product: $\ip{(z,v), (z,w)} = v\bar{w}$.

With respect to these inner products, the symbols  $\sigma_{\tilde E}$, $\sigma_{\tilde{p}}$, and $\sigma_{\tilde{q}}$ of $\tilde{E}$, $\tilde{p}$, and $\tilde{q}$ satisfy
\begin{itemize}
\item When $z \in X$, $f,g \in F_z$, and $\xi \in T^*_z X$
\begin{enumerate}[label={(\arabic*$'$)}]\setcounter{enumi}{4}
\item $\ip{ \sigma_{\tilde E}(\xi) f, f}  \geq 0$ with equality if and only if $\xi = 0$ or $f = 0$,
\item $\ip{ \sigma_{\tilde E}(\xi) f, g} = \ip{ f, \sigma_{\tilde E}(\xi) g}$,
\end{enumerate}
\item When $z \in \partial X$, $f,g \in F_z$, and $\nu \in T^*_z X$ is the cotangent vector with $\nu(T_z \partial X) = 0$ and $\nu(\nabla \rho(z)) = -1$
\begin{enumerate}[label={(\arabic*$'$)}]\setcounter{enumi}{6}
\item $\ip{ \sigma_{\tilde E}(\nu) f, g} = \ip{ \sigma_{\tilde p} f, \sigma_{\tilde p} g}+\ip{ \sigma_{\tilde q}(\nu) f, \sigma_{\tilde q}(\nu) g}$,
\item if $\eta \in T^*_z X$ and $\sigma_{\tilde p} f = \sigma_{\tilde p} g = 0$, then
$$
\ip{ D\sigma_{\tilde E}(\nu; \eta) f, g} = \ip{ \sigma_{\tilde q}(\eta) f, \sigma_{\tilde q}(\nu) g}+\ip{ \sigma_{\tilde q}(\nu) f, \sigma_{\tilde q}(\eta) g}.
$$
\end{enumerate}
\end{itemize}

The next result explains the relationship between $\tilde E$ and $\square$.

\begin{proposition}\label{prop:E versus Square} Suppose $f \in  \mathscr{C}(X;F)$. Then we have:
\begin{enumerate}\renewcommand{\labelenumi}{$(\alph{enumi})$}
\item $f \in {\rm dom}(\bar\partial^*)$ if and only if $\tilde p(f) = 0$. Further, in this case $\bar\partial^* f = \vartheta f$.
\item $\bar{\partial} f \in {\rm dom}(\bar\partial^*)$ if and only if $\tilde q(f) = 0$. Further, in this case $\bar\partial^* \bar \partial f = \vartheta \bar \partial f$.
\end{enumerate}
Hence, if $\tilde p(f) = 0$ and $\tilde q(f) = 0$, then $\square f = \tilde E(f)$.
\end{proposition}

\begin{proof} See for instance the discussion in~\cite[Section 4.2]{Chen-Shaw}. \end{proof}

Using the Morrey--Kohn--H\"ormander identity (see for instance~\cite[Proposition 4.3.1]{Chen-Shaw}), we have the following estimate.

\begin{proposition}\label{prop:MKH} If $f=\sum f_j d\bar{z}^j \in  \mathscr{C}(X;F)$, $\tilde p(f) = 0$, and $\tilde q(f) = 0$, then
$$
\int_{\partial \Omega} \sum_{i,j} \frac{\partial^2 \rho}{\partial z^i \partial \bar{z}^j} f_i \bar{f}_j dS \leq \int_\Omega \ip{\tilde E(f), f} dV.
$$
\end{proposition}

\subsection{Applying Hamilton's result}

Suppose $\Omega \subset \Cb^d$, $\rho : \Cb^d \rightarrow \Rb$, $X : = \overline{\Omega}$, $F \rightarrow X$, $P \rightarrow \partial X$, and $Q \rightarrow \partial X$ are as in the previous section.

Let $M := X \times \Cb^d \rightarrow X$ be the trivial bundle. Then we can identify the space of smooth sections $\mathscr{C}(X; M)$ with the space of smooth maps $C^\infty(X,\Cb^d)$.

 Let $T : C^\infty(X,\Cb^d) \rightarrow C^\infty(\Cb^d, \Cb^d)$ be a linear operator such that for each $k \geq 1$ there exists $C_k > 0$ where
\begin{equation}\label{eqn:bounds on Tf}
\abs{ Tf}_{\Cb^d, k} \leq C_k\abs{f}_{\Omega, k}.
\end{equation}
Such an operator can be found by using the construction in~\cite{Seeley} and a partition of unity.

Then given $m \in \mathscr{C}(X; M)$, define $\psi_m : \Cb^d \rightarrow \Cb^d$ by
$$
\psi_m(z) = z+(Tm)(z).
$$
Using Equation~\eqref{eqn:bounds on Tf}, there exists $\epsilon > 0$ such that: if
$$
\Oc: = \left\{ m \in \mathscr{C}(X; M) : \abs{m}_{\Omega,2} < \epsilon\right\},
$$
then $\Omega_m : = \psi_m(\Omega)$ is a strongly pseudoconvex domain with boundary $\partial\Omega_m : = \psi_m(\partial \Omega)$ for every $m \in \Oc$.

Fix a bounded neighborhood $\Uc_2$ of $\partial \Omega$ where $\nabla \rho \neq 0$ on $\overline{\Uc_2}$. Then fix open sets $\Uc_1, \Uc_3$ such that
$$
\Omega \setminus \Uc_2 \subset \Uc_1 \subset \overline{\Uc_1} \subset \Omega
$$
and
$$
\Cb^d \setminus \Uc_2 \subset \Uc_3 \subset \overline{\Uc_3} \subset \Cb^d \setminus \Omega.
$$
Then fix a partition of unity $\sum_{j=1}^3 \chi_j \equiv 1$,  $\chi_j\geq0$, subordinate to the cover $\Cb^d \subset \cup_{j=1}^3 \Uc_j$.

By shrinking $\Oc$, we may assume that $\partial \Omega_m \subset \Uc_2$ for all $m \in \Oc$. Then
$$
\rho_m : = \chi_1 \left( \rho \circ \psi_m^{-1}\right) + \frac{\chi_2}{\norm{\nabla(\rho \circ \psi_m^{-1})}} \left(\rho \circ \psi_m^{-1}\right) + \chi_3
$$
is a defining function for $\Omega_m$ with $\norm{\nabla \rho_m} \equiv 1$ on $\partial \Omega_m$.

Then let $X_m : = \overline{\Omega}_m$ and let $F_m \rightarrow X_m$, $P_m \rightarrow \partial X_m$, $Q_m \rightarrow \partial X_m$ be the bundles defined in the previous section (where we use $\Omega_m$ in place of $\Omega$ and $\rho_m$ in place of $\rho$). Also let
$$
(\tilde{E}_m, \tilde{p}_m, \tilde{q}_m) : \mathscr{C}(X_m;F_m) \rightarrow  \mathscr{C}(X_m;F_m)\times  \mathscr{C}(\partial X_m;P_m) \times  \mathscr{C}(\partial X_m;Q_m)
$$
be the operators defined in the previous section (where again we use $\Omega_m$ in place of $\Omega$ and $\rho_m$ in place of $\rho$).

Given a vector bundle  isomorphism $\Psi : V \rightarrow W$, let $\Psi_*$ be the induced isomorphism between the space of sections.

\begin{lemma} After possibly shrinking $\Oc$, for every $m \in \Oc$ there exist  $\mathbb C$-linear
  vector bundle isomorphisms $\Psi_m^F : F_0 \rightarrow F_m$, $\Psi_m^P : P_0 \rightarrow P_m$, and $\Psi_m^Q : Q_0 \rightarrow Q_m$   that cover $\psi_m$, $\psi_m|_{\partial \Omega}$, and $\psi_m|_{\partial \Omega}$ respectively.
   Moreover, the coefficients of the linear partial differential operators
$$
E(m) : = (\Psi_m^{F})_*^{-1} \tilde{E}_m (\Psi_{m}^F)_*, \quad p(m) : = (\Psi_m^P)^{-1}_* \tilde{p}_m (\Psi_m^P)_*, \quad q(m) : = (\Psi_m^Q)^{-1}_* \tilde{q}_m (\Psi_m^Q)_*
$$
depend smoothly on $m$ and its derivatives up to degree $3$.
\end{lemma}

\begin{proof} Endow $\Lambda^1(X_m)$ with the standard Hermitian metric on $\Lambda^1(\Cb^d)$. Then let $\pi_{F_m} : \Lambda^1(X_m) \rightarrow F_m$ and $\pi_{Q_m} : \Lambda^1(X_m)|_{\partial X} \rightarrow Q_m$ be the fiberwise orthogonal projections. Then define $\Psi_m^F  = \pi_{F_m} (\psi_m^{-1})^*$, $\Psi_m^P = (\psi_m, \id_{\Cb})$, and $\Psi_m^Q = \pi_{Q_m} (\psi_m^{-1})^*$  respectively on $F_0,P_0,Q_0$. The map $\Psi_m^P$ is clearly a  $\mathbb C$-linear
vector bundle isomorphism. Since $\Psi_0^F = \id_{F_0}$ and $\Psi_0^Q = \id_{Q_0}$, by shrinking $\Oc$ we can assume that each map is a bundle isomorphism for each $m \in \Oc$.

Further, by construction the coefficients of the linear partial differential operators defined above depend smoothly on $m$ and its derivatives up to degree $3$.
\end{proof}

For $m \in \Oc$,
 let $\tilde\nu_m \in \mathscr{C}(\partial X_m; T^* X_m|_{\partial X_m})$ be the covector field with $\tilde \nu_m(T \partial X_m) \equiv 0$ and
$\tilde \nu_m(\nabla \rho_m) \equiv- 1$. Then
$$
 \nu_m:=\psi_m^* \tilde\nu_m = \lambda_m \nu_0
$$
for some smooth function $\lambda_m : \partial X \rightarrow (0,\infty)$.

Next we obtain families of Hermitian metrics, all denoted by $(\ip{\cdot, \cdot}_m)_{m \in \Oc}$, on $F$, $P$, and $Q$ by pulling back using the bundle isomorphisms the standard Hermitian metrics on $\Lambda^1(\Cb^d)$, $\partial X_m \times \Cb$, and $\Lambda^1(\Cb^d)$ respectively.

Let $\mathcal{E} = \{E(m), p(m), q(m)\}_{m \in \Oc}$. By construction this family of operators satisfies conditions (1) to (4) in Section~\ref{appendix:Hamilton's abstract result}.  By~\cite[Thm. 3.9, p.~126]{wells-book}, we have the formula
$$
 (\Psi_m^{F})^{-1} \sigma_{\tilde{E}_m }(\xi)(\Psi_{m}^F)=\sigma_{{E(m)}}(\psi_m^*\xi),
$$
and similar formulas for the other symbols.

Then since the operators $\tilde{E}_m$, $\tilde{p}_m$, $\tilde{q}_m$ satisfy conditions (5') and (6') in Section~\ref{appendix:dbar Neumann} the family of operators $\mathcal{E}$ satisfy conditions (5) and (6) in Section~\ref{appendix:Hamilton's abstract result}.

Also, since the operators $\tilde{E}_m$, $\tilde{p}_m$, $\tilde{q}_m$ satisfy conditions (7') and (8') in Section~\ref{appendix:dbar Neumann} and $\psi_m^*\tilde \nu_m =\nu_m$, the family of operators $\mathcal{E}$ satisfy conditions (7) and (8) in Section~\ref{appendix:Hamilton's abstract result}.  As mentioned early, we can  make $\nu_m$ be $\nu_0$ by scaling the pull-back metrics on $P_0$.

After possibly shrinking $\Oc$, we can assume that
$$
\inf_{m \in \Oc} \min\left\{ \sum_{i,j}  \frac{\partial^2 \rho_m}{\partial z^i \partial \bar{z}^j}(z) v_i \bar{v}_j : z \in \partial X_m, \, v \in \Cb^d, \, \norm{v}=1, \, \sum_j v_j \frac{\partial \rho_m}{\partial z^j}(z) = 0\right\}
$$
is positive (i.e. our family of domains $\{ \Omega_m\}_{m \in \Oc}$ are ``uniformly'' strongly pseudoconvex). We can further assume that the bi-Lipschitz constants of $\varphi_m|_{\partial X_0} : \partial X_0 \rightarrow \partial X_m$ is uniformly bounded. Then by Proposition~\ref{prop:MKH} there exists $C > 0$ such that
$$
\abs{f}_0^2 \leq C\left( {\rm Re} \langle\langle E(m)f,f \rangle\rangle_m + \norm{f}_{0}^2\right)
$$
for all $m \in \Oc$ and $f \in \mathscr{C}(X;F)$ satisfying $p(m)f = 0$ and $q(m)f = 0$.

Thus by Theorem~\ref{Hamilton-result}, there exists a possibly smaller neighborhood $\Oc_0 \subset \Oc$ of the zero section where the family of operators $\mathcal{E}$ has a smooth tame inverse $\mathcal{E}^{-1}$.

\begin{corollary}[{Corollary to~\cite[p.~438]{Hamilton1}}]\label{cor:bergman kernel is smooth in m} For $m \in \Oc_0$, let $K_{\Omega_m}$ be the Bergman kernel on $\Omega_m$. Then the map
$$
(m, z, w) \mapsto K_{\Omega_m}(z,w)
$$
is smooth.
\end{corollary}

\begin{proof} Let $N^m_{0,1} : L_{(0,1)}(\Omega_m) \rightarrow L_{(0,1)}(\Omega_m)$ be the $\bar\partial$-Neumann operator on $\Omega_m$ for $(0,1)$-forms. By Proposition~\ref{prop:E versus Square}, if $\alpha$ is a compactly supported smooth $(0,1)$-form on $\Omega_m$, then
$$
N_{0,1}^m(\alpha) = \left( \Psi^F_m\circ \mathcal{E}^{-1}(m) \circ (\Psi^F_m)^{-1}\right)(\alpha).
$$
Then using Equation~\eqref{eqn:Bergman kernel in terms of dbar Neumann} and the smoothness of $m \mapsto \mathcal{E}^{-1}(m)$, we see that the map
$
(m, z, w) \mapsto K_{\Omega_m}(z,w)
$
is smooth.
\end{proof}

\begin{corollary}\label{Bergman-kernel} If $\{D_t\}_{t \in \Tc}$ is a smooth family of strongly pseudoconvex domains, then the map
$$
(t,z,w) \mapsto K_{D_t}(z,w)
$$
is smooth.
\end{corollary}

\begin{proof} Fix $t_0 \in \Tc$. Then there is a neighborhood $\Tc_0$ of $t_0$ in $\Tc$,  a neighborhood $\mathcal U$ of $\overline{D_{t_0}}$,
 and a $\mathcal C^{\infty}$  smooth map $\Phi : \mathcal U\times\mathcal T_0 \rightarrow \mathbb C^d$ such that for every $t \in \mathcal T_0$ the map $\Phi(\cdot, t) : \mathcal U \rightarrow \Cb^d$ is a diffeomorphism onto its image and $\Phi(\cdot, t)(\overline{D_{t_0}}) = \overline{D_t}$.

 For $t \in \Tc_0$, define $m_t : \overline{D_{t_0}} \rightarrow \Cb^d$ by
 $$
 m_t(z) = \Phi(z, t) - z.
 $$
Then in the notation introduced above, $D_t = (D_{t_0})_{m_t}$. Hence by Corollary~\ref{cor:bergman kernel is smooth in m} we see that
$$
(t,z,w) \mapsto K_{D_t}(z,w)
$$
is smooth in a neighborhood of $t_0$.
\end{proof}

\bibliographystyle{alpha}
\bibliography{complex}

\end{document}